\documentclass[12 pt, letterpaper, oneside]{article}
\usepackage[utf8]{inputenc}
\usepackage{tikz,inputenc,xcolor,amsthm,hyperref, cite, color, url,amsfonts,amsmath,amssymb,caption,multicol,graphicx,fancyhdr,tikz,pgfplots,subcaption,comment,commath,bm}
\usepackage{pdflscape}
\usepgfplotslibrary{groupplots}
\usepackage[toc,page]{appendix}
\usepackage{bbm}

\usepackage[a4paper,bindingoffset=0.5cm,left=1.62715cm,right=1.62715cm,top=2.5cm,bottom=2cm,footskip=.8cm]{geometry}

\DeclareMathOperator{\E}{\mathbb{E}}
\DeclareMathOperator{\Pb}{\mathbb{P}}
\DeclareMathOperator{\N}{\mathbb{N}}

\DeclareMathOperator{\R}{\mathbb{R}}

\renewcommand{\a}{\alpha}
\renewcommand{\b}{\beta}
\renewcommand{\d}{\delta}
\renewcommand{\l}{\lambda}

\newcommand{\m}{\mu}

\newcommand{\e}{\epsilon}

\newcommand{\f}[2]{\frac{#1}{#2}}

\newcommand{\Dt}{\Delta t}
\renewcommand{\L}{\Lambda}

\usepackage{blindtext}
\usepackage{soul}

\newtheorem{theorem}{Theorem}[section]
\newtheorem{lemma}[theorem]{Lemma}
\newtheorem{proposition}[theorem]{Proposition}
\newtheorem{corollary}[theorem]{Corollary}

\newtheorem{conjecture}{\noindent \bf{Conjecture}}[section]
\theoremstyle{definition}
\newtheorem{definition}[theorem]{Definition}

\newtheorem{remark}[theorem]{Remark}
\newtheorem{example}[theorem]{Example}

\newcommand{\pushright}[1]{\ifmeasuring@#1\else\omit\hfill$\displaystyle#1$\fi\ignorespaces}

\usepgfplotslibrary{groupplots}
\usepackage{tikz}
\usepackage{rotating}
\usepackage{amsbsy,enumerate}
\usepackage{graphicx}
\usepackage{comment}
\usepackage{mathrsfs}

\usepackage[affil-it]{authblk}
\title{Infinite-server queues with Hawkes input}
\author[a]{D.T. Koops\footnote{d.t.koops@uva.nl}}
\author[b]{M. Saxena\footnote{m.mayank@tue.nl}}
\author[b]{O.J. Boxma\footnote{o.j.boxma@tue.nl}}
\author[a]{M.  Mandjes\footnote{m.r.h.mandjes@uva.nl}}
\affil[a]{\small Korteweg-de Vries Institute, University of Amsterdam\\P.O.~Box 94248, 1090GE,  Amsterdam.}
\affil[b]{\small Eurandom and Department of Mathematics and Computer Science\\ Eindhoven University of Technology\\ P.O.~Box 513, 5600MB, Eindhoven.}
\date{\small \today}

\newcommand{\vb}{\vspace{3mm}}

\allowdisplaybreaks

\begin{document}
\maketitle

\abstract{\noindent {\footnotesize In this paper we study the number of customers in infinite-server queues with a self-exciting (Hawkes) arrival process. Initially we assume that service requirements are exponentially distributed and that the Hawkes arrival process is of a Markovian nature. We obtain a system of differential equations that characterizes the joint distribution of the arrival intensity and the number of customers. Moreover, we provide a recursive procedure that explicitly identifies (transient and stationary) moments. Subsequently, we allow for non-Markovian Hawkes arrival processes and non-exponential service times. By viewing the Hawkes process as a branching process, we find that the probability generating function of the number of customers in the system can be expressed in terms of the solution of a fixed-point equation. We also include various asymptotic results: we derive the tail of the distribution of the number of customers for the case that the intensity jumps of the Hawkes process are heavy-tailed, and we consider a heavy-traffic regime. We conclude the paper by discussing how our results can be used computationally and by verifying the numerical results via simulations.

\vspace{3mm}

}}

\setlength\parindent{0pt}
\section{Introduction}
\label{sec:intro}
A common assumption in queueing theory is that the customer arrival process is a Poisson process with a deterministic rate. However, various empirical studies have revealed that arrival processes may display \textit{overdispersion}, i.e.\ the variance of the number of arrivals in a given interval exceeds the corresponding expected value, cf.\ e.g.\ \cite{Heemskerk2016} for references. Overdispersion therefore indicates that the standard Poisson assumption (under which the above mentioned variance coincides with the expected value) is not valid. This has led to the study of queueing systems with overdispersed arrival processes \cite{Matthijsen,Koole,Koops2017, Heemskerk2016}. The current paper aims to contribute to this line of research.

\vb

The object of study of this paper is the {\it infinite-server queue} with  {\it Hawkes input}. 
\begin{itemize}
\item[$\circ$]
An infinite-server queue is a service system in which each customer is taken into service immediately upon arrival. Customers are served independently of each other and have i.i.d.\ service requirements. There is a large body of literature on infinite-server queues, but typically some regularity properties are assumed, such as Poisson (or renewal) arrivals.
\item[$\circ$]
{A Hawkes process is a point proces} with an \textit{exogenous} component and an \textit{endogenous} component. The exogenous component generates arrivals according to a homogeneous Poisson process. The endogenous component {entails that arrival epochs coincide with jumps in the arrival rate. The arrival intensity subsequently} behaves deterministically over time according to an \textit{excitation function}. This type of point process was originally studied by Alan Hawkes, cf.\ \cite{Hawkes}, who coined the name \textit{self-exciting process}; nowadays the name {\it Hawkes process} is frequently used as well. 
\end{itemize}
Hawkes processes have been used to model various phenomena, including the firing of neurons in the brain, earthquakes, criminality and riots; see e.g.\ \cite{Bacry} or \cite{Hawkesoverview} for references. Very recently, Hawkes models have also been used to study trending social media \cite{social1,social2,Pender}. Interestingly, Daw and Pender \cite{Pender} add a queueing aspect: they model the arrival process of visitors of a website as a Hawkes process, thus trying to capture the viral behavior of such an arrival process, taking into account that customers leave after a time which has a phase-type distribution. They study the number of visitors on the website at any time $t$. Another queueing application with Hawkes processes is high-frequency transaction processes in limit order books, cf.\ \cite{Bacry,Toke,Contdelarrard2}. Limit order books are in essence queueing systems on financial markets: the books keep track of buy and sell orders, that are waiting until an order arrives that matches the desired execution price above or below a certain limit. It is well known that trades tend to trigger other trades, which makes self-exciting models a natural choice.

Our work also combines queueing with Hawkes arrival processes, an area that is still largely unexplored.  In e.g.\ \cite{Contdelarrard2,Gao2016} scaling limits are derived for queueing systems that allow for Hawkes input. Scaling limits for infinite-server queues designed specifically for Hawkes input are derived in  \cite{Gao2016}; it states that exact and numerical analysis of this model is `challenging'. 
To the best of our knowledge, only \cite{Pender} pays attention to exact (i.e., non-asymptotic) analysis of queues driven by Hawkes processes. In \cite{Pender}, the focus is on exact analysis of an infinite-server queue driven by an unmarked Markovian Hawkes process (i.e., the jump sizes in the arrival intensity are deterministic). The model we consider is a general version of the one studied in [8]: in our case the driving process is a {\it marked} Hawkes process, i.e. the intensity jumps are stochastic. In addition, we use a branching process representation to cover a non-Markovian setting, in which the Hawkes intensity does not need to decay exponentially, and the job distributions are general. Furthermore, we discuss various novel asymptotic results and we show how the obtained results can be used computationally.

Our work is also related to \cite{Koops2017}; the main difference with the model studied there, is that in \cite{Koops2017} the arrival rate
to the infinite-server queue is a Cox process, and hence not self-exciting.
 
\vb

The paper is organized as follows. After having introduced Hawkes processes in Section \ref{sec:self-exciting queue}, we assume in Section \ref{sec:markov} that  the service requirements are exponentially distributed, and the excitation function is an exponentially decreasing function. As a consequence, the number of customers jointly with the Hawkes arrival rate is a Markov process; note that (marginally) the evolution of the number of customers is not Markov. We derive a partial differential equation (PDE) which characterizes the joint Laplace and $z$-transform of the Hawkes intensity and number of customers. We show that transient moments (of arbitrary order) satisfy a specific system of ordinary differential equations (ODEs), and we provide explicit first and second transient moments. Furthermore, by using the characteristic method, we simplify the PDE to a set of ODEs that characterizes the {joint} distribution, which we later use for numerical analysis.

Subsequently, in Section \ref{sec:nonmarkov}, we lift the exponentiality assumptions on the excitation function and the service requirements. The price to be paid comes in the form of less explicit results. Similar to the approach followed in  \cite{HawkesOakes}, the main idea is that the Hawkes process is represented in terms of a branching process, where in our model each node (representing a customer) in this branching process is served in an infinite-server queue. The analysis yields a fixed-point equation for the $z$-transform of the number of customers in the system (Section \ref{sec:nonmarkov}). By performing a finite number of iterations, we can find numerical approximations of the probability mass function of the number of customers,  akin to the approach proposed in \cite{whittfuncinverse}; explicit bounds on the error are derived as well. A numerical example is provided in Section \ref{sec:num}, where the methods from Sections \ref{sec:markov} and \ref{sec:nonmarkov} are verified by simulations as well. The fixed-point equation also leads to new asymptotic results (Section \ref{sec:as}): (i) in the situation that the intensity jumps of the Hawkes process are heavy-tailed, we derive the tail of the distribution of the number of customers, and (ii) we find the asymptotics of the number of customers in the heavy-traffic regime. 
Section~\ref{sec:Discussion} contains suggestions for further research.
%We also propose a numerical method to approximate the probability mass function, akin to the approach of \cite{whittfuncinverse}.  

\section{Preliminaries}
\label{sec:self-exciting queue}
In this section we first formally define the Hawkes process (cf.\ e.g.\ \cite{Bacry,Hawkesoverview,dynamiccontagion} for similar descriptions). Below we consider two equivalent definitions of the Hawkes process, presented in the same generality as they are used in the remainder of this paper.

\begin{definition}[Conditional intensity]
\label{def:Hawkesconditional}
Consider a counting process $(M(t))_{t\geqslant0}$, with associated filtration $({\mathscr F}(t))_{t\geqslant0}$, that satisfies
\[
\Pb(M(t+\Dt)-M(t)=m|{\mathscr F}(t))=
\begin{cases}
\L(t) \Dt + o(\Dt), &m=1\\
o(\Dt),&m>1\\
1-\Lambda(t)\Dt + o(\Dt), &m=0
\end{cases},
\]
as $\Dt\downarrow0$, where the conditional intensity has the form
\begin{equation}\label{Lambda_t}
\L(t) = \l_{\infty}  + \sum_{t_i<t} B_i h(t-t_i),
\end{equation}
where $t_1,t_2,\ldots$ denote arrival epochs, for a set of i.i.d.\ random variables $B_i$ with a nonnegative support, for some {\textit{reversion level}} $\l_{\infty} >0$ and some function $h:[0,\infty)\to[0,\infty)$ which are called the \textit{background intensity} and \textit{excitation function}, respectively. The summand in Eqn.\ \eqref{Lambda_t} is called a \textit{kernel}. The process $M(\cdot)$, as defined above, is called a \textit{self-exciting} or \textit{Hawkes process}.
\end{definition}
Note that an arrival increases {the future arrival intensity}, which in turn increases the probability of another arrival in the future, which explains the name `self-exciting process'.

{
\begin{remark}
There exist definitions of varying generality (cf.\ e.g.\ \cite{Bacry,Hawkesoverview,dynamiccontagion}). For example, there are multidimensional definitions (also referred to as mutually-exciting Hawkes processes), there is a distinction between \textit{marked} and \textit{unmarked} Hawkes processes, and the initial intensity can be taken unequal to the reversion level. The version defined above is considered to be a marked Hawkes process with a multiplicative kernel. It is `marked' because the kernel depends on a random variable $B_i$ (which is called a mark) that is sampled at each event. The kernel is multiplicative, since the size of the increase $B_i$ and the time effect $h(t-t_i)$ are multiplied. Furthermore, note that $\L(0)=\lambda_\infty$ by Eqn.\ \eqref{Lambda_t}. In general, one could consider the case $\L(0)=\l_0\neq\lambda_\infty$. However, this case is hardly more general but introduces more cumbersome notation. Indeed, the additional contribution of the initial intensity can be handled independently and in the same way as the rate increase due to other events.
\end{remark}
}

Next, we consider an alternative definition {of} Hawkes processes, which is based on a representation {of Hawkes processes as branching processes with immigration}. The observation that this is possible was already made in \cite{HawkesOakes}, and it is by now standard in the literature.

\begin{definition}[Cluster representation]
\label{def:Hawkescluster}
Let $\{B_i\}$ be a set of i.i.d.\ random variables with nonnegative support. Consider a (possibly infinite) $T>0$ and define a sequence of events $\{t_m\leqslant T\}$ according to the following procedure:
\begin{itemize}
\setlength\itemsep{0.01em}
\item[$\circ$] Consider a set of \textit{immigrant events} $\{t_m^{(0)}\leqslant T\}$ that arrive according to a homogeneous Poisson process with rate $\l_{\infty} $ in the interval $[0,T]$.
\item[$\circ$]  Set $n=0$. For each arrival labeled by $t_{m'}^{(n)},$ generate a sequence of next generation events $\{t_m^{(n+1)}\leqslant T\}$ in the interval $[t_{m'}^{(n)},T]$, by sampling $B_{m'}$ and then sampling from the resulting Poisson process with time-dependent rate $B_{m'} h(t-t_{m'})$.
\item[$\circ$] Iterate the above rule for $n=1,2,\ldots$, until no more events are generated in $[0,T]$, so as to obtain the event sequence $E_{n}(T):=\{t_m^{(n)}\leqslant T\}$, for $n=0,1,\ldots$.
\end{itemize}
Define $M(t):=|\cup_{n=0}^\infty E_{n}(t)|$ as the number of events in $[0,t]$, then $(M(t))_{0\leqslant t \leqslant T}$ is called a \textit{self-exciting} or \textit{Hawkes process} (on the interval $[0,T]$).
\end{definition}

In the remainder of this paper we will use both definitions to analyze {infinite-server queues driven by Hawkes arrival processes}. Definition \ref{def:Hawkesconditional} plays a key role in Section \ref{sec:markov}, in which we consider Markovian Hawkes queues. The approach in Section \ref{sec:markov} does not apply for non-Markovian queues; in that case we resort to the interpretation as in Definition \ref{def:Hawkescluster}, which is discussed thoroughly in Section \ref{sec:nonmarkov}.

\section{A Markovian Hawkes-fed infinite-server queue}
\label{sec:markov}
In the previous section we introduced the Hawkes process, which serves as the input process of our infinite-server queue. 
In this section we suppose that the arrival process is a Markovian Hawkes process, i.e., the excitation function is exponential and the service requirements are exponentially distributed. In particular, we consider the situation that $h(t)=e^{-r t}$, $r>0$, and the service requirements are $J\sim \exp(\mu)$. In addition, we assume that $B$ (distributed as the $B_i$ featuring in the definition of the Hawkes process) is a random variable such that $\Pb(B>0)=1$. The Hawkes process acts as an input process to an infinite-server system, i.e.\ we could call this model $\textrm{Hawkes}/M/\infty$ in Kendall's notation. 
In Subsection~\ref{sec3.1} we characterize the joint transform of the Hawkes intensity $\Lambda(t)$ and the number of customers $N(t)$
of the $\textrm{Hawkes}/M/\infty$ queue at time $t$, in terms of the solution of an ODE.
In Subsection~\ref{sec:transientmoments} we develop a recursive procedure that gives the transient moments of $(\Lambda(t),N(t))$.
Steady-state moments are briefly discussed in Subsection~\ref{sec:stationarymoments}.

\subsection{Characterization of the queueing process}
\label{sec3.1}
In this section our main objective is to characterize the double transform
\[
\zeta(t,z,s):= \E z^{N(t)}e^{-s\L(t)},
\]
which uniquely defines the joint transient distribution of $(N(t),\Lambda(t))$. 

\begin{theorem}
\label{thm:zeta}
Let the arrival process be a Markovian Hawkes process. Then, given $\L(0)=\l_{\infty} $ and $N(0)=0$,
\begin{equation}
\label{eq:zeta}
\zeta(t,z,s) = e^{-s(t) \l_{\infty} } \exp\left(-\l_{\infty}  r \int_0^t s(u)\dif u\right),
\end{equation}
where{, with $\beta(s):=\E e^{-s B}$}, $s(\cdot)$ solves the ODE
\begin{equation}\label{eq:ODE2}
s'(u)+rs(u) + (1+(z-1)e^{-\mu u})\beta(s(u))-1=0,\quad 0\leqslant u \leqslant t,
\end{equation}
with boundary condition $s(0)=s$. 
\end{theorem}
\begin{remark}
{Note that this result is closely related to results about Hawkes counting processes. For example, a version of \cite[Thm. 3.1]{dynamiccontagion} is retrieved when $\mu=0$ is substituted in Eqn.\ \eqref{eq:ODE2} (note: set $\rho=0$ in \cite{dynamiccontagion}, as we do not consider a shot-noise background process here). Note that $\mu=0$ corresponds to infinitely long service times, and hence the number-of-customers process reduces to a counting process as studied in \cite{dynamiccontagion}. In the situation where $\mu=0$, it was possible to derive an explicit equation for the probability generating function of the counting process by separation of variables. Unfortunately, in our situation the variables $s(u)$ and $u$ in Eqn.\ \eqref{eq:ODE2} are not separable due to the additional factor $e^{-\mu u}$, and hence we cannot solve it analytically.}
\end{remark}

\begin{proof}[Proof of Theorem \ref{thm:zeta}]
We derive the joint distribution {of $(N(t),\Lambda(t))$}. Define
\[
F(t,k,\l)=\Pb(N(t)=k,\Lambda(t)\leqslant \lambda),\quad f(t,k,\l)=\f{\partial F(t,k,\l)}{\partial\l}.
\]
Considering the evolution of the Markovian system between $t$ and $t+\Delta t$ yields
\begin{align*}
F(t+\Dt,k, \l - r(\l-\l_{\infty} )\Dt) &= \int_{0}^\l y \Dt \Pb(B\leqslant \l-y) f(t,k-1,y)\dif y\\&+(k+1)\mu \Dt F(t,k+1,\l)\\
&+F(t,k,\l)(1-k\m\Dt)-\int_{0}^\l y\Dt f(t,k,y)\dif y.
\end{align*}
After elementary manipulations and letting $\Dt\downarrow0$, it follows that
\begin{align*}
\f{\partial F(t,k,\l)}{\partial t} - r(\l-\l_{\infty} )\f{\partial F(t,k,\l)}{\partial \l} &= \int_{0}^\l y \Pb(B\leqslant\l-y)f(t,k-1,y)\dif y \\
& + (k+1)\m F(t,k+1,\l)\\&- k\mu F(t,k,\l)-\int_{0}^\l y f(t,k,y)\dif y.
\end{align*}
Since we assumed that $\Pb(B\leqslant 0)=0$, differentiating with respect to $\l$ yields
\begin{align}
\nonumber
\lefteqn{\f{\partial f(t,k,\l)}{\partial t} - \f{\partial}{\partial\l}[r\l f(t,k,\l)] + {r}\l_{\infty} \f{\partial}{\partial \l}f(t,k,\l) }\\&= \int_{0}^\l y f(t,k-1,y){\dif \Pb(B\leq \lambda-y)}
 + (k+1)\mu f(t,k+1,\l) - (k\mu+\l) f(t,k,\l).\label{eq:pde}
\end{align}
The next step consists of transforming Eqn.\ \eqref{eq:pde} {with respect to the Hawkes intensity $\l$}, and to that end we define the transform
\[
\xi(t,k,s):=\int_{0}^\infty e^{-s\l} f(t,k,\l)\dif\l.
\]
{After transformation we find}
\begin{align*}
&\f{\partial \xi(t,k,s)}{\partial t} + rs \f{\partial\xi(t,k,s)}{\partial s} + rs\l_{\infty} \xi(t,k,s)\\
&= -\f{\partial\xi(t,k-1,s)}{\partial s} \beta(s)+(k+1)\m\xi(t,k+1,s)-k\m\xi(t,k,\l)+\f{\partial\xi(t,k,s)}{\partial s},
\end{align*}
which we can rewrite as the partial differential equation
\begin{align*}
\lefteqn{\hspace{-2cm}
\f{\partial \xi(t,k,s)}{\partial t} + (rs-1) \f{\partial\xi(t,k,s)}{\partial s}+\f{\partial\xi(t,k-1,s)}{\partial s} \beta(s) + rs\l_{\infty} \xi(t,k,s) }\\&= (k+1)\m\xi(t,k+1,s)-k\m\xi(t,k,\l).
\end{align*}
Next, we transform this equation in the number of customers variable $k$, for which we use the transform
\[
\zeta(t,z,s)= \sum_{k=0}^\infty z^k \xi(t,k,s)=\E z^{N(t)}e^{-s\L(t)}.
\]
This yields
\begin{equation}
\label{eq:stronghawkesPDE}
\f{\partial \zeta(t,z,s)}{\partial t} + (rs+z\b(s)-1) \f{\partial\zeta(t,z,s)}{\partial s} + \m(z-1)\f{\partial\zeta(t,z,s)}{\partial z} =-rs\l_{\infty} \zeta(t,z,s).
\end{equation}

Now let $s$ and $z$ be parametrized by $u$. Then the characteristic equations are
\begin{align}
\label{eq:sz}
-s'(u)+rs(u) + z(u)\beta(s(u))-1 &= 0,\\ 
-z'(u)+\mu(z(u)-1) &= 0\label{eq:z},
\end{align}
with the boundary conditions $s(t)=s$ and $z(t)=z$. The solution of Eqn.\ \eqref{eq:z} is 
\[
z(u)= 1+ C e^{\mu u},
\]
with $C$ determined by $z(t) = z = C e^{\mu t} + 1$, i.e., $C= (z-1) e^{-\mu t}$.
We thus find that 
$
z(u)= 1+ (z-1) e^{-\mu(t-u)},
$ which we can substitute in Eqn.\ \eqref{eq:sz}, so as to obtain the ODE
\begin{equation*}
-s'(u)+rs(u) + (1+(z-1) e^{-\mu(t-u)}) \beta(s(u))-1 = 0.
\end{equation*}
By substituting $t$ for $t-u$, we obtain the ODE
\begin{equation*}
s'(t)+rs(t) + (1+(z-1)e^{-\mu t})\beta(s(t))-1=0,
\end{equation*}
with the boundary condition $s(0)=s$. The result follows.
\end{proof}

Numerically, one can obtain the probability mass function of $N(t)$ by first solving the differential equation, and then applying a Fourier inversion algorithm. Numerical results are given in Sec.\ \ref{sec:num}.
Another powerful feature of Theorem \ref{thm:zeta} is that it allows us to find moments, as is presented in Sections \ref{sec:transientmoments}--\ref{sec:stationarymoments} below.

\subsection{Transient moments}
\label{sec:transientmoments}
{In this section we discuss the computation of joint transient moments of the type $\E \Lambda^g(t)N^q(t)$, for some integers $g$ and $q$, under the additional assumption that $\E B^g < \infty$}. The marginal moments of the Hawkes process are known and can be found in e.g.\ \cite[Lemma 3.1 \& Thm.\ 3.6]{dynamiccontagion},  but we include them here for completeness.  The key idea is to use the PDE \eqref{eq:stronghawkesPDE} to derive ODEs for the joint moments. 
To this end, rewrite the PDE\ \eqref{eq:stronghawkesPDE} as
	\begin{align}\label{Eq: Transiernt1}
  &\f{{\rm{d}} }{{\rm{d}}  t} \E z^{N(t) }e^{-s\L(t)}	- (rs+z\b(s)-1) \E \Lambda(t) z^{N(t) }e^{-s\L(t)} + \m(z-1) \E N(t) z^{N(t) - 1}e^{-s\L(t)}  \nonumber \\ &=-rs\l_{\infty}  \E z^{N(t)} e^{-s\L(t)}.
	\end{align}
	We begin by differentiating the above equation $g\in{\mathbb N}$ times with respect to $s$ and then inserting $s = 0$, which expresses the $(g + 1)^{\rm th}$ moment of $\Lambda(t)$ in terms of the first up to $g^{\rm th}$ moment of $\Lambda(t)$, through the following ODE (where we have assumed that $\E B^g < \infty$):
	\begin{align}\label{Eq: Transiernt11}
	 \f{{\rm{d}} }{{\rm{d}}  t} \E \Lambda^{g}(t) z^{N(t)}  &+ g (r - z\,\E B )\E \Lambda^g(t) z^{N(t)}  +  \m (z - 1) \E \Lambda^g(t) N(t) z^{N(t) - 1}\nonumber \\ 
	&=  (z - 1) \E \Lambda^{g + 1}(t) z^{N(t)} + \mathbbm{1}_{\{g \geqslant 1\}}~  g \l_{\infty}  r \E \Lambda^{g - 1}(t) z^{N(t)}  \nonumber \\ &+ z\  \mathbbm{1}_{\{g \geqslant 2\}}  \sum_{j = 0}^{g - 2} {{g}\choose{j}} \E B^{g - j} ~ \E \Lambda^{j + 1}(t) z^N(t), ~ g = 0, 1, \dots ~ .	
%	& - \f{{\rm{d}} }{{\rm{d}}  t} \E \Lambda^{g}(t) z^{N(t)} + (z - 1) \E \Lambda^{g + 1}(t) z^{N(t)} - g (r - z\,\E B )\E \Lambda^g(t) z^{N(t)}  \nonumber \\ 
%	&=  \m (z - 1) \E \Lambda^g(t) N(t) z^{N(t) - 1} - g \l_{\infty}  r \E \Lambda^{g - 1}(t) z^{N(t)} \nonumber \\ &- z\  \mathbbm{1}_{\{g \geqslant 2\}}  \sum_{j = 0}^{g - 2} {{g}\choose{j}} \E B^{g - j} ~ \E \Lambda^{j + 1}(t) z^N(t).
	\end{align}

	So as to obtain a relation between the joint moments and the moments that correspond to number of customers, we differentiate Eqn.\ \eqref{Eq: Transiernt11} $q + 1$ times with respect to $z$ and insert $z = 1$. This gives us the key relation, for $g=0,1,\dots$, $q=0,1,\dots$:
\begin{align}\label{Eq: Transiernt2}
& \f{{\rm{d}} }{{\rm{d}}  t} \E \Lambda^{g}(t) \bar{N}^q(t) + ((q + 1)  \m + g(r - \E B)) \E \Lambda^g(t) \bar{N}^{q}(t) \nonumber \\ &=
(q + 1)\E \Lambda^{g + 1}(t) \bar{N}^{q - 1}(t) + \mathbbm{1}_{\{g \geqslant 1\}}~ g \l_{\infty}  r\, \E \Lambda^{g - 1}(t) \bar{N}^{q}(t) + (q + 1) g\, \E B\, \E \Lambda^g(t) \bar{N}^{q - 1}(t) \nonumber \\ & + \mathbbm{1}_{\{g \geqslant 2\}}  \sum_{j = 0}^{g - 2} {{g}\choose{j}}  \E B^{g - j}   \left[ (q + 1) \E \Lambda^{j + 1}(t) \bar{N}^{q - 1}(t) + \E \Lambda^{j+1}(t) \bar{N}^{q}(t)\right],
\end{align}
where $\bar{N}^{q}(t) := N(t) (N(t) - 1) \cdots (N(t) - q)$ and $\bar{N}^{-1}(t) := 1$.

Now by \eqref{Eq: Transiernt2} we can obtain a system of first order ODE\,s. To this end, we  substitute in \eqref{Eq: Transiernt2} a combination of indices,  $g := k$, $q := q - k$ for values of  $k \in\{0,\ldots, q + 1\}$. Denoting \[Z^{(q+2)}(t) := \big[\,\E \bar{N}^q(t), \E \L(t) \bar{N}^{q - 1}(t), \dots, \E \L^q(t) \bar{N}^0(t), \E \L^{q + 1}(t)\,\big]^{\rm T}, \] it follows that  the vector $Z^{(q + 2)}(t), q = 0, 1, 2, \dots,$ satisfies the ODE
		\begin{equation}\label{Main_ODE_Moemnts}
		\f{{\rm d}}{{\rm d}t} Z^{(q+2)}(t) =  A_1^{(q+2)} Z^{(q+2)}(t) +  A_0^{(q+2)},
		\end{equation}
		with 
		\begin{align}		
		 A_1^{(q+2)} &= \begin{bmatrix}
		- d^{(q + 1)}_{0} \!& q + 1& 0 &  \cdots & 0 & 0 \\
		0 & - d^{(q)}_{1}  & q &  \cdots & 0 & 0 \\
		0 & 0 & - d^{(q - 1)}_{2} &  \cdots & 0 & 0\\
		\vdots & \vdots & \vdots &   \ddots & \vdots & \vdots \\
		0 & 0 & 0  &  \cdots & - d^{(1)}_{q}  & 1\\
		0 & 0 & 0 &   \cdots &  0 & -d_{q + 1}^{(0)}
		\end{bmatrix}, 	~~	
		 A_0^{(q+2)} = \begin{bmatrix}
		b^{(q)}_0 \\ 
		b^{(q-1)}_1 \\ 
		b^{(q-2)}_2 \\  
		\vdots \\ 
		 b^{(0)}_q \\ 
		b^{(-1)}_{q+1}
		\end{bmatrix},
		\end{align}
		where, for $k \in\{0,\ldots, q + 1\}$, 
		\begin{align*}
				d^{(q + 1 - k)}_k &:= (q + 1 - k) \mu + k ({r- \E B}), \\ 
				b^{(q - k)}_k &:= \mathbbm{1}_{\{k \geqslant 1\}} ~ k \l_{\infty}  r \,\E \Lambda^{k - 1}(t) \bar{N}^{q - k}(t) + (q - k + 1) k \E B \,\E \Lambda^k(t) \bar{N}^{q - k - 1}(t)  \\ &+ \mathbbm{1}_{\{k \geqslant 2\}}  \sum_{j = 0}^{k - 2} {{k}\choose{j}}  \E B^{k - j}   \left[ (q - k + 1) \E \Lambda^{j + 1}(t) \bar{N}^{q - k - 1}(t) + \E \Lambda^{j+1}(t) \bar{N}^{q - k}(t)\right].
		\end{align*}
		Note that  $d^{(q+ 1)}_{0},\ldots, d^{(0)}_{q+1}$ are the (distinct) eigenvalues of the matrix $ A_1^{(q+2)}$.

	\begin{proposition}\label{All_Moments} 
		The solution of the ODE \eqref{Main_ODE_Moemnts} is 
		\begin{equation}\label{General_Moments_Sol}
		Z^{(q+2)}(t) = e^{A_1^{(q+2)} t} \ Z^{(q+2)}(0) + \int_{0}^{t}  e^{A_1^{(q+2)} (t - s)}  A_0^{(q+2)} \ {\rm d} s,
		\end{equation}
		where, with  $I$  a $(q + 2) \times (q + 2)$ identity matrix,
		\begin{equation}\label{Exp_Matrix}
		    e^{A_1^{(q+2)}t}= \sum_{i = 1}^{q + 2} e^{d^{q + i}_{i - 1} t}\prod_{\substack{j=1,\\j\neq i}}^{q+2} \f{A_1^{(q + 2)} - d^{(q - j + 2)}_{j - 1} I}{d^{(q - i + 2)}_{i - 1} - d^{(q - j + 2)}_{j - 1}}.
		\end{equation} 
	\end{proposition}
	\begin{proof} Formula \eqref{General_Moments_Sol} follows from \eqref{Main_ODE_Moemnts} in a straightforward way. The exponential of the matrix is computed by using the {interpolation-based} formula given in  \cite[p. 101, Eqn. (4.18)]{ghufran2010computation}.
		%\cite[p. 827, Eqn. (1.4)]{dehghan2010computing}.
	\end{proof}
	
	\begin{corollary}
	\label{cor:momentsLambda}
	If $r\neq b_1:=\E B$, and with $r_0:=r-b_1$, it holds that 
	\begin{align*}
	\E \L(t) &= \frac{\l_{\infty}  r }{r_0} - \frac{\l_{\infty}  b_1}{r_0} e^{-r_0t}, \\
	\E N(t)  &= \frac{\l_{\infty}  r}{\m r _0} - \frac{\l_{\infty}  b_1}{r_0(\m -r_0)} e^{- r_0t} - \f{\l_{\infty}  (r - \m)}{\m(\m -r_0)} e^{-\m t},
	\end{align*}	
	with an obvious adaptation when $\m = r_0$. 
\end{corollary}
\begin{proof}
		To get the first moments of the model we can use Proposition \ref{All_Moments} for the case $q = 0$:
	\begin{equation}\label{first_moments_ODE}
	\f{{\rm d}}{{\rm d}t}
	Z^{(2)}(t) =  A_1^{(2)} Z^{(2)}(t) +  A_0^{(2)},
	\end{equation}
	with 
	\[Z^{(2)}(t) = \begin{bmatrix}
	\E N(t) \\ 
	\E \L(t)
	\end{bmatrix}, \:\:\: A_1^{(2)} = \begin{bmatrix}
	- \m & 1 \\
	0 & - r_0	\end{bmatrix}, ~\mbox{and}~\: A_0^{(2)} = \begin{bmatrix}
	0 \\ \l_{\infty}  r 
	\end{bmatrix}.	\]
We know the solution of the general ODE \eqref{Main_ODE_Moemnts} from Eqn. \eqref{General_Moments_Sol}. Therefore, for the case $q = 0$ we get the solution to the ODE \eqref{first_moments_ODE}:
\begin{equation}\label{General_first_moments}
 \begin{bmatrix}
\E N(t) \\ 
\E \L(t)
\end{bmatrix} =   e^{A_1^{(2)} t} \begin{bmatrix}
\E N(0) \\ 
\E \L(0)
\end{bmatrix}+ \int_{0}^{t} e^{A_1^{(2)} (t - s)}   A_0^{(2)} {\rm d} s,
\end{equation}
where $(\E N(0),\E \L(0)) = (0,\l_{\infty} )$. Now it remains to evaluate  $e^{A_1^{(2)}t}$. By virtue of \eqref{Exp_Matrix}, we find
\[
e^{A_1^{(2)} t} = \begin{bmatrix}
e^{- \m t} & {\displaystyle \f{e^{-\m t} - e^{- r_0 t}}{-\m + r_0}} \\
&\vspace{-2mm}\\
0 & e^{- r_0t}
\end{bmatrix}.
\]
Substituting this exponential matrix in \eqref{General_first_moments} yields the desired results. 
\end{proof}

	We now explicitly state the second transient moments. Define $b_2:=\E B^2$.
\begin{corollary}
	\label{cor:momentsN}
	If $r_0\neq 0$, then for all $t\geqslant0$,
	\begin{align*}
	 \E \L^2(t) &=  \frac{\l_{\infty}  r (b_2 + 2\l_{\infty}  r)}{2 r_0 ^2} - \frac{\l_{\infty}  b_1(b_2 + 2\l_{\infty}  r)}{r_0 ^2} e^{- r_0 t} + D_1 e^{-2 r_0 t}, \\	
	\E \L(t) N(t)  &= \frac{\l_{\infty}  r}{r_0 (\m + r_0)} \left( \frac{b_2 + 2 \l_{\infty}  r}{2 r_0 } + \frac{\l_{\infty}  r + \m b_1}{\m} \right) - \frac{\l^2_0 ~r ~(r - \m)}{\m r_0 (\m -  r_0 )} e^{-\m t}  \nonumber \\ &	- \frac{\l_{\infty}  b_1}{\m r_0 } \left( b_1 + \frac{b_2 + 2\l_{\infty}  r}{r_0} + \frac{\l_{\infty}  r}{\m -  r_0 } \right) e^{-  r_0 t} \nonumber \\
	& + \frac{D_1}{\m -  r_0 } e^{-2 r_0 t}  	+ D_2 e^{-(\m + r_0) t},\\
	\E N^2(t)  & =\frac{\l_{\infty}  r}{\m r_0(\m + r_0)} \left( \frac{b_2 + 2 \l_{\infty}  r}{2 r_0 } + \frac{\m(\m + r) + \l_{\infty}  r}{\m} \right) -  \f{\l_{\infty}  (r - \m)(\m r_0  + 2 \l_{\infty}  r)}{\m^2 r_0 (\m -  r_0 )} e^{-\m t} \\ & - \frac{2 \l_{\infty}  b_1}{\m r_0  (2 \m -  r_0 )} \left( b_1 + \frac{b_2 + 2\l_{\infty}  r}{r_0} + \frac{\l_{\infty}  r}{\m -  r_0 } + \f{\m}{2}\f{2 \m -  r_0 }{\m -  r_0 }\right)  e^{-  r_0 t} \\ &
		+ \frac{D_1  }{(\m -  r_0 )^2} e^{-2 r_0 t}   + \frac{2 D_2}{\m - r_0 } e^{-(\m + r_0) t} + D_3 e^{-2\m t},
	\end{align*}
	where the constants $D_1, D_2$ and $D_3$ follow from the initial conditions, i.e., the requirements that $\E N^2(0) = 0, \E \L(0) N(0) = 0$, and $\E \L^2(0) = \l_{\infty} ^2$.
	
\end{corollary}
\begin{proof}
	To obtain all second moments we   again use Proposition \ref{All_Moments}, for $q = 1$ (with $\bar{N}^0(t) = N(t)$). This yields 
	\begin{equation}\label{Second_moments_ODE}
	\f{{\rm d}}{{\rm d}t} Z^{(3)}(t)
 =  A_1^{(3)} Z^{(3)}(t)  +  A_0^{(3)}(t),
	\end{equation}
	with 
	\begin{align*} Z^{(3)}(t) &= \begin{bmatrix}
	\E \bar{N}^1(t) \\ 
	\E \L(t) N(t) \\ 
	\E \L^2(t)
	\end{bmatrix}, \:\:\: A_1^{(3)} = \begin{bmatrix}
	- 2 \m & 2 & 0 \\
	0 & - \m -  r_0  & 1 \\
	0 & 0 & - 2  r_0 
	\end{bmatrix}, \\  A_0^{(3)}(t) &= \begin{bmatrix}
	0 \\ b_1 \E \L(t) + \l_{\infty}  r \E N(t) \\ (b_2 + 2 \l_{\infty}  r) \E \L(t)
	\end{bmatrix}.	\end{align*} 
	
	Again using Eqn. \eqref{General_Moments_Sol} for the case $q = 1$ yields the solution to ODE \eqref{Second_moments_ODE}:
	\begin{equation}\label{Second_moments_sol}
	\begin{bmatrix}
	\E \bar{N}^1(t) \\ 
	\E \L(t) N(t) \\ 
	\E \L^2(t)
	\end{bmatrix} =   e^{A_1^{(3)} t} \begin{bmatrix}
	\E \bar{N}^1(0) \\ 
	\E \L(0) N(0) \\ 
	\E \L^2(0)
	\end{bmatrix} + \int_{0}^{t} e^{A_1^{(3)} (t - s)}   A_0^{(3)}(s) {\rm d} s,
	\end{equation}
	where
	\[\begin{bmatrix}
	\E \bar{N}^1(0) \\ 
	\E \L(0) N(0) \\ 
	\E \L^2(0)
	\end{bmatrix} = \begin{bmatrix} 	0 \\  0 \\   \l^2_0	\end{bmatrix}.
	\]
	Now we only need to calculate matrix $e^{A_1^{(3)}t }$ to get the complete solution of \eqref{Second_moments_ODE}. Using \eqref{Exp_Matrix}, we get
	% \cite{van1996matrix},
%	\begin{align*}
%	e^{A_1^{(3)}t } = {\displaystyle \begin{bmatrix}
%	e^{-2 \mu t } & {\displaystyle  \frac{2   (e^{-(\mu + r_0 )t}- e^{-2 \mu t }   )}{\mu-r_0 }} & \frac{\displaystyle e^{-2 \mu t } +e^{- 2  r_0  t}-2 e^{-(\mu +r_0 )t }}{\displaystyle (\mu-r_0
%	 )^2} \\
%	&&\vspace{-2mm}\\
%	0 & e^{-(\mu + r_0  )t} & \frac{\displaystyle e^{- 2  r_0  t}-e^{-(\mu +r_0 ) )t}}{\displaystyle \mu-r_0 } \\ &&\vspace{-2mm}\\
%	0 & 0 & e^{- 2  r_0  t}
%	\end{bmatrix}}.
%	\end{align*}
	
	\begin{align*}
	e^{A_1^{(3)}t } = {\displaystyle \begin{bmatrix}
		e^{-2 \mu t } & {\displaystyle  \frac{2 e^{-\mu t}  (e^{-r_0 t}- e^{-\mu t }   )}{\mu-r_0 }} & \frac{\displaystyle (e^{-r_0 t}- e^{-\mu t })^2}{\displaystyle (\mu-r_0
			)^2} \\
		&&\vspace{-2mm}\\
		0 & e^{-(\mu + r_0  )t} & \frac{\displaystyle  e^{-r_0 t}(e^{-r_0 t} - e ^{-\mu t})}{\displaystyle \mu-r_0 } \\ &&\vspace{-2mm}\\
		0 & 0 & e^{- 2  r_0  t}
		\end{bmatrix}}.
	\end{align*}
For $\m = r_0$, this expression has to be adapted using the l'Hospital rule. Since we now know the first factorial moment $\E \bar{N}^1(t)$ and the first moment $\E N(t) $, the  relation  $\E N^2(t) = \E \bar{N}^1(t) + \E N(t)$ provides the second transient moment. We do not write the constants $D_1,D_2$ and $D_3$ explicitly, because their expressions are rather lengthy.
\end{proof}

\begin{remark}
Note that in Corollaries \ref{cor:momentsLambda} and \ref{cor:momentsN}, the speed of convergence to the stationary moments of $N(\infty)$, and of ${\rm Cov}(\Lambda(\infty),N(\infty))$, is determined by the minimum of $r_0$ and $\mu$.
\end{remark}

\begin{remark}
\label{rem:stationary}
The results in Corollaries \ref{cor:momentsLambda} and \ref{cor:momentsN} are stated under the condition that $r\neq{b_1}$. The case in which $r={b_1}$ (i.e., $r_0=0$) can be derived from those results, if one takes the limit $r\to{b_1}$. From the expressions in Corollary \ref{cor:momentsN}, note that if $r-{b_1}=r_0<0$, then $\E N(t)$ tends to infinity as $t\to\infty$. It turns out that the condition $r-{b_1}=r_0>0$ leads to finite moments, and can be interpreted as the model's stability condition. This stability condition could also have been derived a priori, by considering the Hawkes process as a branching process (as in Definition \ref{def:Hawkescluster}). From branching process theory, it is known that a `population' (or in our case, a cluster) exterminates with probability one if and only if the average number of children per parent is less than unity \cite{branchingbook}. In this case, this translates to the condition
\[
{b_1} \int_0^\infty h(s)\dif s = {b_1} \int_0^\infty e^{-rs}\dif s = \f{{b_1}}{r}<1.
\]
If this does not hold, then in the long run there will be more and more clusters, which implies that the intensity and hence the number of customers will grow indefinitely on average.  
Observe that the value of the service rate $\mu$ does not affect stability as long as $\mu > 0$. 
\end{remark}

\subsection{Stationary moments}
\label{sec:stationarymoments}
In this subsection we consider stationary moments. To this end, we assume that the queueing process is stable, i.e.\ ${b_1}<r$; cf.\ Remark \ref{rem:stationary}. Stationary moments can be easily derived from the previous section by taking the limit $t\to\infty$. Define by $\Lambda$ and $N$ the stationary versions of $\Lambda(t)$ and $N(t)$, respectively. 

\begin{corollary}
\label{cor:stationarymomentslambda}
If $\E B^g < \infty$, then
the $g^{\rm th}$ moment of the stationary Hawkes intensity is given by
\begin{equation*}\label{true_hawkes_moments}
\E \Lambda^g = \frac{1}{g r_0}\left[\mathbbm{1}_{\{g \geqslant 1\}} ~ g \l_{\infty}  r \E \Lambda^{g - 1}  + \mathbbm{1}_{\{g \geqslant 2\}}  \sum_{j = 0}^{g - 2} {{g}\choose{j}} {\E B}^{g - j} ~\E \Lambda^{j + 1}\right], ~~ g = 1, 2, 3, \dots
\end{equation*}
\end{corollary}
\begin{proof}
Considering Eqn. \eqref{Eq: Transiernt11} in the steady-state case and then taking $z = 1$ yields the desired result.
\end{proof}

%\begin{corollary}
%	\label{cor:staedymomentsN}
%	The first and second moments of the number of customers stationary Hawkes intensity is given by,
%	\begin{align*}
%	\E N^2(t)  & =\frac{\l_{\infty}  r}{\m r_0(\m + r_0)} \left( \frac{b_2 + 2 \l_{\infty}  r}{2 r_0 } + \frac{\m(\m + r) + \l_{\infty}  r}{\m} \right), \nonumber \\
%	\E \L(t) N(t)  &= \frac{\l_{\infty}  r}{r_0 (\m + r_0)} \left( \frac{b_2 + 2 \l_{\infty}  r}{2 r_0 } + \frac{\l_{\infty}  r + \m b_1}{\m} \right) ,\\
%   \E N(t)  &= \frac{\l_{\infty}  r}{\m r _0}.
%	\end{align*}
%\end{corollary}

\begin{corollary}
Let $b_1$ and, where needed, $b_2$ be finite.
The stationary means and (co-)variances of $\Lambda$ and $N$ are given by:
\begin{align*}
\E \Lambda &= \frac{\l_{\infty}  r}{r _0}, \:\:\ \E N = \frac{\l_{\infty}  r}{\m r _0}, \:\:\
\mathbb{V}{\rm ar}(\L)=  \frac{\l_{\infty}  r b_2}{2r_0^2}, \\ \mathbb{V}{\rm ar}(N) &=\frac{\l_{\infty}  r}{2 r_0 ^2} \f{b_2  + 2(\m + r) r_0 }{\m(\m +  r_0 )},\:\ ~\mbox{and}~\:\ \mathbb{C}{\rm ov}(N, \L) = \f{\l_{\infty}  r}{2r_0^2}\f{b_2 + 2 r_0  {b_1}}{\m +  r_0 } .
\end{align*} 

%\begin{align*}
%\mathbb{C}{\rm ov}(N, \L) &= \f{\l_{\infty}  r}{2r_0^2}\f{b_2 + 2 r_0  {b_1}}{\m +  r_0 },\:\:\:
%\mathbb{V}{\rm ar}(N) =\frac{\l_{\infty}  r}{2 r_0 ^2} \f{b_2  + 2(\m + r) r_0 }{\m(\m +  r_0 )},\:\:\
%\mathbb{V}{\rm ar}(\L)=  \frac{\l_{\infty}  r b_2}{2r_0^2}\\
% \E N &= \frac{\l_{\infty}  r}{\m r _0}, 
% ~\mbox{and}~
%  \E \Lambda = \frac{\l_{\infty}  r}{r _0}.
%\end{align*} 

In particular, the correlation coefficient of $N$ and $\L$ is given by
\[
\rho(N, \L) = \f{b_2 + 2 r_0  {b_1}}{\sqrt{b_2 (b_2  + 2 r_0 (\m + r))}}\sqrt{\f{\m}{\m +  r_0 }} .
\]
\end{corollary}
\begin{proof}
Follows directly from taking $t\to\infty$ in Corollaries \ref{cor:momentsLambda} and \ref{cor:momentsN}. 
\end{proof}

\section{A non-Markovian Hawkes-fed infinite-server queue}
\label{sec:nonmarkov}
In this section we allow the excitation function $h$ to be a general nonnegative function, and we allow $J$ to be any {nonnegative} random variable. 
%We denote this model by $\textrm{Hawkes}/G/\infty$. 
We derive a fixed-point equation for the $z$-transform of $N(t)$  in Theorem \ref{prop:N}, which  can be iterated so as to obtain an (increasingly accurate) approximation; Proposition \ref{prop:upperlower} derives error bounds of the resulting approximation scheme.

%{\subsection{Cluster analysis}}\footnote{{Since now this is the only subsection in this section, we might want to delete this.}}
%\label{sec:cluster}
This section relies heavily on the representation of a Hawkes process as a branching process (recall Definition \ref{def:Hawkescluster}). Immigrants arrive according to a Poisson process with rate $\l_{\infty} $. Each of those immigrants increases the future arrival rates. The arrivals that occur due to this increase, are called the {\it children} of the immigrant. In turn, those children are potentially the parents of a next generation, and so forth. Since the customers enter an infinite-server queue, each of them is served independently of the rest, with i.i.d.\ service requirements (distributed as a nonnegative random variable $J$). 

Note that, by definition, each parent produces children independently from all other parents at a rate $B h(u)$ at time $u$ after its own birth, conditional on its own arrival rate increment $B$. Let $S(u)$ denote the number of children of a parent, $u$ time units after its own birth, including the parent itself if it is still being served. Let its probability generating function be denoted by $\eta(u,z):= \E z^{S(u)}$, for $0\leqslant u\leqslant t$. Then the number of jobs in the system at time $t$, $N(t)$ with $N(0)=0$, satisfies 
\begin{equation}
\label{eq:cluster}
\E z^{N(t)} = \sum_{n=0}^\infty\f{(\l_{\infty}   t)^n}{n!}e^{-\l_{\infty}   t} \left(\f1t\int_0^t \eta(u,z)\dif u\right)^n = \exp\left(\l_{\infty}   \int_0^t (\eta(u,z)-1)\dif u\right),
\end{equation}
which follows by conditioning on the number of immigrants (i.e., the number of clusters). The next step is to identify $\eta(u,z)$, by studying each cluster separately.
First consider the distributional equality, for $0\leqslant u\leqslant t$,
\begin{equation}
\label{eq:cdf}
S(u)\stackrel{\textrm{d}}{=} 1_{\{J>u\}} + \sum_{i=1}^{K(u)}S^{(i)}(u-t_i),
\end{equation}
where $S(u)\stackrel{\textrm{d}}= S^{(i)}(u)$ for all $i$, $K(\cdot)$ is an inhomogeneous Poisson counting process with rate $B_i h(\cdot)$ (conditional on $B_i$) that counts the number of children, and $t_1,t_2,\ldots$ are the birth times of the corresponding children. Note that $S^{(i)}(u)$ can be interpreted as the number of children of child $i$ (including itself if it is still in the system, $u$ time units after its birth). Denote by $P_t(s)$ the probability that, {conditional} on the fact that a child was born before time $t$, it was already born before time $s$. Then it holds that {
\begin{align*}
P_t(s)&=\f{\Pb(K(s)=1,K(t)-K(s)=0)}{\Pb(K(t)=1)}
=\f{e^{-\int_0^s B h(u)\dif u}\int_0^s B h(u)\dif u e^{-\int_s^t B h(u)\dif u}}{e^{-\int_0^t B h(u)\dif u}\int_0^t B h(u)\dif u} = \f{\int_0^s h(u)\dif u}{\int_0^t h(u)\dif u}.
\end{align*}}
Defining {$H(u):=\int_0^u h(v)\dif v$}, the corresponding PDF is thus given by{
\[
p_{ u}(s) := P'_u(s) =  \f{h(s)}{\int_0^u h(v)\dif v} = \f{h(s)}{H(u)}.
\]}
Define ${\mathscr J}(u):={\mathbb P}(J > u)$ and $\bar {\mathscr J}(u):=1 - {\mathscr J}(u)$. Then it follows that 
\begin{align*}
&\eta(u,z)=\E_B\left[\sum_{n=0}^\infty \E[z^{S(u)}|K({ u})=n,B]\Pb(K({ u})=n|B)\right]\\
&=( \bar {\mathscr J}(u)+{\mathscr J}(u)z)  \E_B\left[\sum_{n=0}^\infty \E[z^{S^{(1)}(u)}]^n e^{-B H({ u})}\f{B^n H({ u})^n}{n!}\right]\\
&=( \bar {\mathscr J}(u)+{\mathscr J}(u)z) \E_B\left[\sum_{n=0}^\infty \left(\int_0^u p_{ u}(s)  \eta(u-s,z) \dif s\right)^n e^{-B H({ u})}\f{B^n H({ u})^n}{n!}\right]\\
&=( \bar {\mathscr J}(u)+{\mathscr J}(u)z) \E_B \exp\left(B \int_0^u h(s)( \eta(u-s,z)-1)\dif s\right).
\end{align*} 
Recognizing the Laplace-Stieltjes transform of $B$, we can write
\begin{equation}
\label{eq:fixedpoint}
 \eta(u,z)=( \bar {\mathscr J}(u)+{\mathscr J}(u)z) \beta\left(\int_0^u h(s)(1- \eta(u-s,z))\dif s\right).
\end{equation}
Let ${\mathscr G}$ be the class of all time-dependent $z$-transforms $f(u,z):=\E z^{X(u)}$, where $X(u)$, for $u\in[0,t]$, is a nonnegative discrete random variable living on the integers, with a possibly defective distribution (i.e., the probabilities may sum to strictly less than one). For $f\in{\mathscr G}$, define the functional $\phi$ by
\begin{equation}
\label{eq:functional}
\phi(f)(u,z)=( \bar {\mathscr J}(u)+{\mathscr J}(u)z) \beta\left(\int_0^u h(s)(1- f(u-s,z))\dif s\right).
\end{equation}
With this definition, Eqn.\ \eqref{eq:fixedpoint} can be summarized as $\eta=\phi(\eta)$, i.e., $\eta$ can be seen as a fixed point of $\phi$. We stress that the following result involves \textit{complex} $z$. This is important, because for numerical analysis in Section~\ref{sec:num} we need to iterate $\phi$ a finite number of times for a particular set of complex-valued arguments. Furthermore, we will use the following notation. For a generic $f_0\in\mathscr{G}$ we define $f_{n+1}:=\phi(f_n)$, for $n=0,1,\ldots$.

\begin{theorem}
\label{prop:N}
Suppose that $\E B\int_0^\infty h(s)\dif s<\infty$. Let $z\in\mathbb{C}$, such that $|z|<1$, be fixed. Then the following holds:
\begin{enumerate}
\item[(i)] the $z$-transform of $N(t)$ is given by Eqn.\ \eqref{eq:cluster};
\item[(ii)] the function $\eta$, defined by $\eta(u,z)=\E z^{S(u)}$ for $u\in[0,t]$ is the unique fixed point of $\phi$, i.e.\ $\phi(\eta)=\eta$, with $\phi$ defined in Eqn.\ \eqref{eq:functional};
\item[(iii)] the sequence $(f_n)_n$ has the property that $f_n\in\mathscr{G}$ for every $n$ if $f_0\in\mathscr{G}$;
\item[(iv)] $f_n$ converges pointwise to $\eta\in\mathscr{G}$, regardless of the initial approximation $f_0\in\mathscr{G}$.
\end{enumerate} 
\end{theorem}
\begin{proof}
First fix $-1<z<1$, where $z$ is a real number. Note that (i) and $\phi(\eta)=\eta$ (cf.\ Eqn.\ \eqref{eq:fixedpoint}) have already been proven. To prove (ii), it remains to show uniqueness, which will be shown after the proof of (iii) and (iv). It holds that $f_0\in{\mathscr G}$ implies $f_1=\phi(f_0)\in{\mathscr G}$, following a similar probabilistic reasoning as in e.g.\  \cite[Thm.\ 1]{whittfuncinverse}. Indeed, consider the related operator $\tilde{\phi}$, which maps a possibly defective CDF $F$ to
\begin{equation}
\label{eq:tildephi}
\tilde\phi(F)(u,k)=\Pb\left(1_{\{J>u\}} + \sum_{i=1}^{K(u)} S^{(i)}(u-t_i)\leqslant k\right),\quad 0\leqslant u\leqslant t,
\end{equation}
cf.\ also Eqn.\ \eqref{eq:cdf}, where the $S^{(i)}$ are i.i.d.\ with CDF $F$. Let $f_0\in{\mathscr G}$ be the $z$-transform of a possibly defective CDF $F$. Note that then $\phi(f_0)$ is the $z$-transform of the CDF $\tilde\phi(F)$, by construction of $\phi$ and $\tilde\phi$. Statement (iii) now follows by induction.

\vb

Now we show that, for arbitrary $f_0,g_0\in\mathscr{G}$, $f_n$ and $g_n$ have a limit, which is in fact the same for both sequences. To do this, we show by induction that for every $n\in\mathbb{N}$,
\begin{equation}
\label{eq:induction}
|f_n(u,z)-g_n(u,z)|{\leqslant} \f{1}{n!}(Cu)^n,\quad 0{\leqslant} u{\leqslant} t.
\end{equation}
We first prove the base case $n=1$, i.e.\ for every $0{\leqslant} u{\leqslant} t$,
\begin{align}
\label{eq:chain}
\nonumber |f_1(u,z)-g_1(u,z)|&\stackrel{(1)}{\leqslant}\E\left|e^{-B\int_0^u h(s)(1-f_0(u-s,z))\dif s}-e^{-B\int_0^u h(s)(1-g_0(u-s,z))\dif s}\right|\\
&\stackrel{(2)}{\leqslant} \E B \int_0^u h(u-s)|f_0(s,z)-g_0(s,z)|\dif s\\
&\nonumber \stackrel{(3)}{\leqslant} \E B \int_0^u h(s)\dif s \int_0^u|f_0(s,z)-g_0(s,z)|\dif s\stackrel{(4)}{\leqslant} C u,
\end{align}
with $C=2\E B \int_0^\infty h(s)\dif s$, which is finite by assumption. Inequality (1) follows from Eqn.~\eqref{eq:functional} and $|\bar{\mathscr J}(u)+{\mathscr J}(u)z|\leqslant 1$ for $|z|<1$. Inequality (2) follows by the mean-value theorem, which implies that for $x,y\in\R$, $e^x-e^y = (x-y)e^{z(x,y)}$, where $z(x,y)$ takes a value between $x$ and $y$. In our application, the exponent on the right-hand side can be bounded from above by $1$, since $f_0,g_0\in\mathscr{G}$ implies that $f_0,g_0\leqslant 1$, and hence the exponential is bounded by $e^0=1$. Finally, inequality (3) follows from Young's inequality for convolutions \cite[Thm.  3.9.4]{Young}, and (4) from $|f_0-g_0|\leqslant 2$, for $f_0,g_0\in\mathscr{G}$, and $0\leqslant\int_0^u h(s)\dif s \leqslant \int_0^\infty h(s)\dif s$ since $h\geqslant 0$. Now consider the inductive step: if Eqn.\ \eqref{eq:induction} holds for some $n\in\mathbb{N}$, then by the same reasoning as in Eqn. \eqref{eq:chain}, for $0\leqslant u\leqslant t$,
\begin{align*}
|f_{n+1}(u,z)-g_{n+1}(u,z)|&{\leqslant} \E B\int_0^u h(s)\dif s \int_0^u |f_n(s,z)-g_n(s,z)|\dif s\\
&{\leqslant} C \f{1}{n!} \int_0^u (Cs)^n \dif s = \f{1}{(n+1)!} (Cu)^{n+1}.
\end{align*}
It follows that Eqn.\ \eqref{eq:induction} holds for every $n\in\mathbb{N}$. As $n$ tends to infinity, we can conclude that $f_n$ and $g_n$ have the same limit $f$, regardless of the initial approximations $f_0,g_0$. Moreover, it turns out that the limit $f$ is an element of $\mathscr{G}$, which follows from $f_n\in\mathscr{G}$ for every $n$ in combination with L\'evy's convergence theorem \cite[Ch. XVIII]{WILL}. From the definition of $\phi$, it is easy to see that $\phi$ is a continuous operator, and hence
\[
f= \lim_{n\to\infty} f_{n+1} =\lim_{n\to\infty} \phi(f_n) = \phi\Big(\lim_{n\to\infty} f_n\Big) = \phi(f), 
\] 
thus the limit $f\in\mathscr{G}$ is a fixed point of $\phi$, and hence (iv) is proved. The above now also immediately yields uniqueness of the fixed point of $\phi$, which finishes the proof of (ii). Indeed, suppose that $\eta_0,\tilde \eta_0\in\mathscr{G}$ are both fixed points of $\phi$, then
\[
\eta_0=\lim_{n\to\infty}\eta_n=\lim_{n\to\infty}\tilde\eta_n =\tilde\eta_0.
\] 
Finally, we need that the above not only holds for real $z$, but also complex $z$ such that $|z|<1$. The conditions in \cite[Thm. 5]{whittfuncinverse} hold\footnote{Technically these theorems are about Laplace-Stieltjes transforms, but they can be adapted to a result for $z$-transforms by applying a variable transform.}, so that we can conclude that $f_n\to\eta$, for all \textit{complex} $z$ with $|z|<1$, regardless of the initial approximation $f_0\in{\mathscr G}$. 
\end{proof}

The iterates of $z$-transforms can be used to construct lower and upper bounds on the CDF of $S(u)$. We write $[n]$ in the superscript of a mapping to denote an $n$-fold application of the mapping. Define $F^\ast\equiv 1$, $F_\ast\equiv 0$, $F^{n\ast}=\tilde\phi^{[n]}(F^\ast)$ and $F_{n\ast} = \tilde\phi^{[n]}(F_\ast)$.

\begin{proposition}
\label{prop:upperlower}
For every $0\leqslant u\leqslant t$ and for all $n,k\in\N$,
\[
1=F^\ast(u,k)\geqslant F^{n\ast}(u,k)\geqslant F^{(n+1)\ast}(u,k)\geqslant F(u,k)\geqslant F_{(n+1)\ast}(u,k)\geqslant F_{{n}\ast}(u,k)\geqslant F_{\ast}(u,k) =0,
\]
where $F_{n\ast}$ and $F^{n\ast}$ are the (defective) CDF\,s associated with $\phi^{[n]}(z\mapsto 1)$ and $\phi^{[n]}(z\mapsto 0)$, respectively.
\end{proposition}
\begin{proof}
Due to  \cite[Thm. 8]{whittfuncinverse}, it is sufficient to show that $\tilde{\phi}$ as defined in Eqn.\ \eqref{eq:tildephi}, regarded as an operator mapping possible defective CDF\,s into possibly defective CDF\,s is monotone in the stochastic ordering. In other words, if we define $F_1\leqslant_{\textrm{st}}F_2$ to mean $F_1(k)\geqslant F_2(k)$ for all $k\in\N$, then we need to show that $F_1\leqslant_{\textrm{st}}F_2$ implies $\tilde\phi(F_1)\leqslant_{\textrm{st}}\tilde\phi(F_2)$. If $F_1\leqslant_{\textrm{st}}F_2$, then we can construct i.i.d.\ random variables $\{S_j^{(i)}:i\geqslant1\}$, for $j=1,2$, such that $S_j^{(1)}$ has CDF $F_j$, for $j=1,2$, and $S^{(i)}_1\leqslant S^{(i)}_2$ for each $i$. Hence, for every $0\leqslant u\leqslant t$,
\[
1_{\{J>u\}} + \sum_{i=1}^{K(u)} S_1^{(i)}(u-t_i) \leqslant 1_{\{J>u\}} + \sum_{i=1}^{K(u)} S_2^{(i)}(u-t_i)
\]
with probability one, which implies that, for every $0\leqslant u\leqslant t$,
\begin{align*}
    \tilde\phi(F_1)(u,k)&=\Pb\left(1_{\{J>u\}} + \sum_{i=1}^{K(u)} S_1^{(i)}(u-t_i)\leqslant k\right)\\
&\geqslant\Pb\left(1_{\{J>u\}} + \sum_{i=1}^{K(u)} S_2^{(i)}(u-t_i)\leqslant k\right) = \tilde\phi(F_2)(u,k),
\end{align*}
for all $k$. The result follows.
\end{proof}

\begin{remark}
\label{rem:alg}
We start in the extremal functions $0$ and $1$ in $\mathscr{G}$, to ensure that the associated (defective) CDFs provide an upper and lower bound to the CDF of the solution. However, this gives the worst possible bounds and leaves some room for improvement. To save computation time, another practical issue is to find a good initial approximation $f_0\in\mathscr{G}$. One may for instance use the approximation $e^x\approx 1+x$ in Eqn.\ \eqref{eq:functional}. For example, in the case of exponential jobs, the resulting functional equation is solvable using techniques from \cite[Chapter I]{tricomi1957integral}. The solution will provide a better initial approximation than simply the function identical to 1 or 0. In this paper, however, we decide not to pursue these issues any further.
\end{remark}

\section{Numerics and simulations}
\label{sec:num}
The purpose of this section is to show the practical applicability of the tools developed in the preceding sections. At the same time, this exercise verifies the main distributional results of this paper. To be more precise, we apply the Markovian method (as developed in Section \ref{sec:markov}) and the cluster-based approach  (as developed in Section \ref{sec:nonmarkov}). In order to be able to compare the results of both methods, we  work with the more restrictive assumptions mentioned in the beginning of Section \ref{sec:markov}.

\vb

Some comments about both approaches are in order. Firstly, to use Theorem \ref{thm:zeta} for numerical results, we need to numerically invert the transform in Eqn.\ \eqref{eq:zeta}. We do this by using a fast Fourier inversion algorithm, called \textsc{Poisson}, which was published in \cite{ztrans}. This algorithm requires the evaluation of Eqn.\ \eqref{eq:zeta} for complex values. In each instance, this leads to solving an ODE with a complex boundary, which can be accomplished by standard ODE solvers. Secondly, with regards to the functional equation \eqref{eq:functional}, in case that $J=\infty$ (i.e., for the usual Hawkes process without departures), \cite{HawkesOakes} commented that such relations are `rather intractable'. Also in \cite{Gao2016}, it is mentioned that the problem of finding the probability mass function of the number of customers is generally `numerically challenging'. We will now show,    however,  that Eqn.\ \eqref{eq:functional} is in fact tractable. Indeed, we again use the \textsc{Poisson} algorithm for the branching process approach, for which we need to evaluate Eqn.\ \eqref{eq:cluster} for complex values. Since $\eta$ is only implicitly known, the idea is now that for each complex argument required by \textsc{Poisson}, we iterate Eqn.\ \eqref{eq:functional} a number of times, until it is `close enough' to its fixed point. We have verified that convergence occurs for complex $z$ in Section \ref{sec:nonmarkov}, cf.\ Theorem \ref{prop:N}. Furthermore, we draw upon Proposition~\ref{prop:upperlower} to derive numerical upper and lower bounds. The upper and lower bounds provide a maximum error that is made by using a finite number of iterations of the functional equation in Theorem~\ref{prop:N}. 

\vb 

Proposition~\ref{prop:upperlower} provides upper and lower bounds on the CDF, while we are actually interested in the probability mass functions. Clearly, upper and lower bounds on the CDF can be used to derive upper and lower bounds on the PMF. For example, for $k\geqslant1$,
\[
\left.
\begin{array}{l}
a_0\leqslant\Pb(X\leqslant k)\leqslant b_1\\
a_1\leqslant\Pb(X\leqslant k+1)\leqslant b_2
\end{array}
\right\}\implies
a_1-b_1\leqslant\Pb(X=k+1)\leqslant b_2-a_0.
\]

\vb

As a third way to validate the results we perform simulations. The simulation algorithm is directly based on Definition~\ref{def:Hawkescluster}, in combination with the thinning procedure of \cite{sim} for inhomogeneous Poisson processes. For closely related simulation algorithms, cf.\ \cite{ogatasim}. After simulating the Hawkes arrival proces, we simply flip a coin with success probability $\Pb(J>t-t_i)$, for arrivals entering at time $t_i$, to determine if they are still in the system at time $t$.

\begin{example}
In the fast Fourier inversion algorithm, in this example, we used an accuracy of $\gamma=4$, with $\gamma$ defined in \cite{ztrans}. Suppose that we have an exponential excitation function, with rate $r$, and exponential job sizes with mean $1/\mu$. Then we are in the Markovian setting and we can use the PDE method from Section \ref{sec:markov}, as well as the cluster approach from Section \ref{sec:nonmarkov}. In this example, we take the shot sizes $B$ to be deterministic. We choose the parameters as given in Table \ref{tab:parameters}. The parameters are arbitrarily chosen, however, one should realize that increasing $t$ leads to a higher computation time for both methods. Increasing $B$ or decreasing $r$ leads to a higher computation time only for the method of Section \ref{sec:nonmarkov}, given a fixed accuracy. The results are given in Table \ref{tab:results}. The accuracy of the results is determined by the discretization step in solving the ODE, which is set at $10^{-4}$, and by the number of iterations in the functional approach, which is taken to be 10, and the discretization size of the numerical integrals, which we took to be $2^{-12}$. Those parameters are chosen in such a way that the computation time is reasonable (in the order of a couple of minutes), and such that finer discretization does not improve the results substantially. 

We performed 100 batches of 100\,000 simulations (totalling 10 million runs), and calculated the standard deviations over those 100 batches (runtime approximately 1 hour). Note that the estimated values of all methods agree quite well for small values of $N(10)$ in Table \ref{tab:results}. For larger values of $N(10)$, the simulation estimates become more uncertain, as well as the spread between the upper and lower bound of the cluster method. Higher accuracy in the tail can be achieved by performing more than the current number of iterations.

{
\begin{table}[ht!]
\footnotesize
\captionsetup{font=small}
  \centering
  \begin{tabular}{ccccc}
    $t$ & $\l_{\infty} $ & $r$ & $\mu$ & $B$\\
    \hline
    10&1.45&2.15&1.25&0.98\\
  \end{tabular}
    \caption{\small \textit{Parameter choice, where $B$ is deterministic.}}
    \label{tab:parameters}
\end{table}

\begin{table}[ht!]
\footnotesize
\captionsetup{width=0.8775\textwidth}
\captionsetup{font=small}
  \centering
  \begin{tabular}{c|ccccccccc}  
	$\Pb(N(10) = \cdots)$&0&1&2&3&4&5&6\\
	\hline
Cluster upper&1.83e-1&2.54e-1&2.19e-1&1.51e-1&
9.22e-2&5.24e-2&2.87e-2\\
Cluster lower&1.83e-1&2.53e-1&2.17e-1&1.48e-1&
8.91e-2&4.89e-2&2.49e-2\\
Cluster &1.83e-1&2.54e-1&2.18e-1&1.50e-1&
9.08e-2&5.07e-2&2.68e-2\\
Diff.\ Eqn.&1.83e-1&2.54e-1&2.18e-1&1.50e-1
&9.09e-2&5.09e-2&2.70e-2\\
Simulations&1.83e-1&   2.54e-1&   2.18e-1&   1.50e-1&
   9.10e-2&   5.09e-2&   2.70e-2\\  
Sim st.\ dev.&0.01e-1&   0.01e-1&   0.01e-1&   0.01e-1&   0.09e-2&   0.07e-2&   0.05e-2\\ 
\hline
\multicolumn{8}{}{}\\
\hline
$\Pb(N(10) = \cdots)$&7&8&9&10&11&12&13\\
\hline
Cluster upper&1.56e-2&8.69e-3&5.23e-3&3.53e-3&
2.73e-3&2.34e-3&2.16e-3\\
Cluster lower&1.17e-2&4.74e-3&1.23e-3&0&0&
0&0\\
Cluster &1.36e-2&6.73e-3&3.24e-3&1.53e-3&7.12e-4&
3.27e-4&1.48e-4\\
Diff.\ Eqn.&1.38e-2&6.81e-3&{3.29e-3}&{1.56e-3}&{7.20e-4}
&{3.34e-4}&{1.53e-4}\\
Simulations&   1.37e-2& 6.8e-3&   3.3e-3&   1.6e-3&   7.4e-4&
   3.4e-4 &1.6e-4\\
Sim st.\ dev.&0.04e-2&  0.3e-3&   0.2e-3&   0.1e-3&   0.8e-4&   0.6e-4&   0.4e-4\\
\hline
\end{tabular}
    \caption{\textit{This table lists the probabilities that there are $0,1,\ldots, 13$ customers in the system. The standard deviation of the simulations is listed in the row `Sim st.\ dev'. The rows cluster upper, lower, respectively refer to the upper and lower bound obtained by using Proposition~\ref{prop:upperlower}. The row `Cluster' uses the initial approximation $f_0\equiv 1$. The row `Diff.\ Eqn.' is the solution that relies on Theorem~\ref{thm:zeta}.}} 
    \label{tab:results}
\end{table}}
\end{example}

\section{Asymptotic results}
\label{sec:as}
In this section we use the findings from the previous section to derive asymptotic results.
In Subsection~\ref{sec5.1} we assume that $B$ is regularly varying of index $-\alpha$.
We show that $N(t)$ is then also regularly varying, of the same index.
In Subsection~\ref{sec5.2} we study the heavy-traffic behavior of $N(\infty)$.
In both cases we strongly rely on the representation (\ref{eq:cluster}).

\subsection{Heavy-tailed asymptotics}
\label{sec5.1}
In this subsection we consider the case that the random variables $B_i$ are heavy-tailed. We use the following definition of a regularly varying random variable.
\begin{definition}
	A random variable $X$ on $[0, \infty)$ is called regularly varying of index $-\alpha$, denoted by ${\mathscr R}(-\alpha$), with $\alpha >0$, if
	\begin{equation}\label{RVDef}	
	\mathbb{P}(X > x) = \ell(x) x^{-\alpha},~ {x \geqslant 0},
	\end{equation}
	with $\ell(x)$ a slowly varying function at infinity, i.e., ${\ell(\gamma x)}/{\ell(x)} \to 1$ as $x\to\infty$ for all $\gamma > 1.$	
%We also say that the random variable $B$ is regularly varying of index $-\alpha$.
\end{definition}

We now prove that, if the $B_i$ are ${\mathscr R}(-\alpha)$, with $1< \alpha < 2$, then so is the number of customers
$N(t)$ at time $t$ in the
$\textrm{Hawkes}/G/\infty$ queue.
Notice that this contrasts with the fact that the number of customers in the $M/G/\infty$ queue (starting empty at time 0) is Poisson distributed at any time $t$.

{
In the sequel $\star$ denotes convolution, e.g.:
\[(f\star g)(u) := \int_0^u f(s) g(u-s) {\rm d}s,\]
and $n$-fold convolutions are iteratively defined by
\[
f^{n\star}(u) = (f\star f^{(n-1)\star})(u), \quad \text{for } n\geq 2,\quad\text{with }f^{1\star}\equiv f.
\]
In addition we define}, with $1 < \alpha < 2$,
\begin{align}
	R_1(u) &:= \sum_{n = 0}^{\infty} (b_1)^n ~(h^{n\star}\star{\mathscr J})(u), ~~ u > 0,
\label{nr25-1}
\\
	R_{\alpha}(u) &:= \Gamma(1-\alpha) ~ \ell(\infty)  ~  \sum_{n = 0}^{\infty}(b_1)^n \,  \left(h^{n\star}\star  ( h\star R_1)^\alpha\right)(u) , ~~ u > 0.
\label{nr25-2}
	\end{align}
	We throughout impose the stability condition, which now reads (cf.\ Remark~\ref{rem:stationary}): $\rho:=b_1 \int_0^\infty h(s){\rm d}s<1$.

\begin{theorem}\label{Theo:Heavy_tail} 
	If $B$ is ${\mathscr R}(-\alpha)$ with $\alpha \in (1, 2)$, then so is $N(t)$: as $z \uparrow 1$,
	\begin{equation}\label{Eq:Heavy_tail}
	\E z^{N(t)} - 1 + \l_{\infty}  (1 - z) \int_{0}^{t} R_1(u) {\rm d}u \sim  - \l_{\infty}  (1 - z)^{\alpha} \int_{0}^{t}  R_{\alpha}(u) {\rm d}u.
	\end{equation}
    \end{theorem}

\begin{proof}We use Theorem \ref{prop:N} to characterize the $z$-transform   of the distribution of $N(t)$.   It follows from \cite[Thm.~8.1.6]{bingham1989regular}, which relates the behavior of a regularly varying function at infinity and the behavior of its {Laplace-Stieltjes transform (LST)} near $0$, that  $\beta(s) - 1 +s\, {b_1}  \sim - \Gamma(1-\alpha) ~ s^{\alpha} ~\ell ({s}^{-1}),$ as $s  \downarrow 0$, for $\alpha \in (1, 2)$. Thus, as $z \uparrow 1$,
	\begin{align}\label{HT1}
		&\beta\left( \int_{0}^{u} h(s)\left(1 - \eta(u - s, z)\right)  {\rm{d}} s\right) - 1 + \left( \int_{0}^{u} h(s)\left(1 - \eta(u - s, z)\right)  {\rm{d}} s\right)   {b_1} \nonumber \\  & \sim  -\Gamma(1-\alpha) ~ \left(\int_{0}^{u} h(s)\left(1 - \eta(u - s, z)\right)  {\rm{d}} s \right)^{\alpha} \ell\left(\f{1}{\int_{0}^{u} h(s)\left(1 - \eta(u - s, z)\right)  {\rm{d}} s }\right).	
  	\end{align}
  	Substituting this into Eqn.\ \eqref{eq:fixedpoint}, we obtain, as $z \uparrow 1$,
  	\begin{align}\label{HT1}
  	\lefteqn{1 - \eta(u, z)     \sim 1 - \left(\bar{\mathscr J}( u) + z {\mathscr J}( u)\right) \Bigg[ 1 - \left( \int_{0}^{u}  h(s)\left(1 - \eta(u - s, z)\right)  {\rm{d}} s \right){b_1}  \nonumber }\\		
  	&   -{\Gamma(1-\alpha)} ~ \left(\int_{0}^{u} h(s)\left(1 - \eta(u - s, z)\right)  {\rm{d}} s \right)^{\alpha} \ell\left(\f{1}{\int_{0}^{u} h(s)\left(1 - \eta(u - s, z)\right)  {\rm{d}} s }\right) \Bigg].	
  	\end{align}
	Rearranging the terms of (\ref{HT1}) and simplifying yields, as $z \uparrow 1$, up to $O((z - 1)^2)$ terms,
	\begin{align}\label{HT2}
	&1 - \eta(u, z) \nonumber \\  &=   1 - (1 - {\mathscr J}(u)   (1 - z)) \Bigg[ 1 -\left( \int_{0}^{u}  h(s)\left(1 - \eta(u - s, z)\right)  {\rm{d}} s \right) {b_1}  \nonumber \\		
	&  - {\Gamma(1-\alpha)} ~ \left(\int_{0}^{u} h(s)\left(1 - \eta(u - s, z)\right)  {\rm{d}} s \right)^{\alpha} \ell\left(\f{1}{\int_{0}^{u} h(s)\left(1 - \eta(u - s, z)\right)  {\rm{d}} s }\right)\Bigg]   \nonumber \\
	&= {\mathscr J}( u)\, (1 - z) + \left( \int_{0}^{u}  h(s)\left(1 - \eta(u - s, z)\right)  {\rm{d}} s \right){b_1}\nonumber \\		
	& + {\Gamma(1-\alpha)} ~ \left(\int_{0}^{u} h(s)\left(1 - \eta(u - s, z)\right)  {\rm{d}} s \right)^{\alpha} \ell\left(\f{1}{\int_{0}^{u} h(s)\left(1 - \eta(u - s, z)\right)  {\rm{d}} s }\right).
	\end{align}
{Observe that}
 $1 - \eta(u, 1) =  0$, and if $ \E S(u)$ is finite, then $1 - \eta(u, z) = 1 - \E z^{S(u)} =  \E S(u) (1-z) + o(1 - z)$. Now we claim  that the leading term in  $1 - \eta(u, z)$ in \eqref{HT2} should be $\E S(u) (1 - z)$ with $\E S(u)=R_1(u)$,
as given in (\ref{nr25-1}),
for $z \uparrow 1$. To this end, observe that it cannot be $R_1(u) (1 - z)^{1 + \epsilon}$ for some $\epsilon > 0$ because of the term ${\mathscr J} ( u) (1 - z) $ in the right hand side of Eqn.\ \eqref{HT2} --- realize that if that would be the case then dividing by $(1 - z)$ and letting $z \uparrow 1$ yields a contradiction. So the dominant term of $1 - \eta(u, z)$ for $z \uparrow 1$, is $R_1(u) (1 - z)$, where
	\begin{align}\label{HT3}
		R_1(u) &= {\mathscr J }( u) + b_1 \int_{0}^{u} h(s) R_1( u -  s)  {\rm{d}} s.	
	\end{align}	
	
	The above equation can be recognized as  a Volterra integral equation of the second kind  (see \cite[Ch. I]{tricomi1957integral}). By using Picard iteration, we obtain, for $u\geqslant 0$,
	\begin{align}\label{HT4}
	R_1(u) &= {\mathscr J} (u) + {b_1} \int_{0}^{u} h(s){\mathscr J} ( u - s)   {\rm{d}} s \nonumber \\
	&+ {b_1} \int_{0}^{u} h(s) \left[ {b_1} \int_{0}^{u - s} h(v) {\mathscr J}( u - s - v) {\rm{d}} v \right]   {\rm{d}} s + \cdots \nonumber \\
	&= {\mathscr J}( u) + {b_1} \, (h\star{\mathscr J})( u)  + ({b_1})^2 \, (h^{2\star}\star{\mathscr J})(u)   + \cdots \nonumber \\
	& = \sum_{n = 0}^{\infty} ({b_1})^n ~(h^{n\star}\star {\mathscr J})(u), ~~ u > 0.
	\end{align}
	Now we claim when $z \uparrow 1$  (using the same reasoning as for the first term) that the next term of $1 - \eta(u, z)$ is $R_{\alpha}(u) (1 - z)^{\alpha}$, with $R_{\alpha}(u)$ given in (\ref{nr25-2}).  Substituting this in \eqref{HT2} and calculating the coefficient of $(1 - z)^{\alpha}$, we find	in the limit $z \uparrow 1$:
	\begin{align}\label{HT5}
	R_{\alpha}(u) &=   {b_1} \int_{0}^{u} h(s) R_{\alpha}(u - s) {\rm{d}} s + {\Gamma(1-\alpha)} ~ \ell(\infty) ~  \left(\int_{0}^{u} h(s) R_1(u - s) {\rm{d}} s\right)^{\alpha}
.
%l\left(\f{1}{\int_{0}^{u} h(s) R_1(u - s) (1 - z)  {\rm{d}} s}\right).
	\end{align}
	This equation is also a Volterra integral equation of the second kind (see \cite[Ch. I]{tricomi1957integral}) and using   Picard iteration, we get 
	\begin{align}
	R_{\alpha}(u) \nonumber &= {\Gamma(1-\alpha)} ~ \ell(\infty) ~  \sum_{n = 0}^{\infty}  ({b_1})^n ~  \left(h^{n\star}\star  ( h\star R_1)^\alpha\right)(u), ~~ u > 0.
%l\left(\f{1}{\int_{0}^{u} h(s) R_1(u - s) (1 - z)  {\rm{d}} s}\right)\right).
	\end{align}	
	
	Now we use Theorem~\ref{prop:N}, in combination with the asymptotics that we just established for the first and second term. As $z \uparrow 1$,
	\begin{align}\label{HT_final}
	\E z^{N(t)}
%\exp\left(\l_{\infty}  \int_{0}^{t} (\eta(u, z) - 1) {\rm d}u\right)
 %\nonumber  \\
	&\sim \exp\left(-\l_{\infty}  \int_{0}^{t} ( R_1(u) (1 - z)  + R_{\alpha}(u) (1 - z)^\alpha) {\rm d}u\right) \nonumber  \\	
	&= \exp\left(-\l_{\infty}  (1 - z) \int_{0}^{t} R_1(u) {\rm d}u\right) \exp\left(-\l_{\infty}  (1 - z)^\alpha \int_{0}^{t} R_{\alpha}(u) {\rm d}u\right).
	\end{align}
	Using the exponential function expansion in the above equation, we get
	\begin{align}\nonumber
	&\E z^{N(t)} \nonumber \\ &\sim  \left(1 -\l_{\infty}  (1 - z) \int_{0}^{t}  R_1(u) {\rm d}u + O((1 - z)^2)\right)\\
	&\nonumber\hspace{3cm}\cdot
	\left(1 - \l_{\infty}  (1 - z)^\alpha \int_{0}^{t} R_{\alpha}(u) {\rm d}u + O((1 - z)^{2 \alpha})\right) \nonumber \\
	&=  1 - \l_{\infty}  (1 - z) \int_{0}^{t} \left(R_1(u) + (1 - z)^{\alpha - 1}  R_{\alpha}(u) \right) {\rm d}u + O((1 - z)^{2}). \nonumber
	\end{align}
	
Using \cite[Thm.~8.1.6]{bingham1989regular} in the reverse way, we conclude that, indeed, $N(t) \in {\mathscr R}(-\alpha)$, with $\alpha \in (1,2)$.
\end{proof}

\begin{remark}
For special choices of $h(\cdot)$ and $\mathscr J (u)$, one can obtain explicit expressions for $\int_0^t R_1(u) {\rm d}u$ and
$\int_0^t R_{\alpha}(u) {\rm d}u$. 
Below we consider the case $h(u) = {\rm e}^{-ru}$ and $\mathscr J(u)  = {\rm e}^{-\mu u}$.
Substituting these in Eqn. \eqref{HT3}, we get
	\begin{equation}\label{GT_Htaffic20}
	R_1(u) = e^{-\mu u} + b_1 \int_{0}^{u} e^{-r (u - s)} R_1(s)  {\rm{d}} s.
	\end{equation}
	The above equation is a Volterra integral equation of the second kind with kernel $k(u, s) = e^{-r(u - s)}$. Solving it using the resolvent kernel method given in \cite[Chapter I]{tricomi1957integral}, we get
	\[
	R_1(u) = \f{1}{r_0 - \m} \left[(r - \m)e^{- \m u} - b_1 \ e^{-r_0 u}\right].
	\]
From this, or directly from (\ref{GT_Htaffic20}) after integrating both sides from $0$ to $\infty$, we obtain
\[
\int_0^{\infty} R_1(u) {\rm d}u = \frac{1/\mu}{1 - \rho} .
\]
	We next turn to $R_{\alpha}(u)$. We substitute $h(\cdot)$ and $R_1(\cdot)$ in Eqn. \eqref{HT5}, which leads to a Volterra integral equation of the second kind, i.e.,	
	\[
	R_{\alpha}(u) = b_1 \int_{0}^{u} e^{-r(u - s)} R_{\alpha}(s)  {\rm{d}} s + {\Gamma(1-\alpha)} ~ \ \ell(\infty)  \left(\f{e^{-\m u} - e^{-r_0u}}{r_0 - \m}\right)^\alpha.
	\]
Substituting $r_0 = r (1 - \rho)$ and then integrating both sides from $0$ to $\infty$ yields
\[
\int_0^{\infty} R_{\alpha}(u) {\rm d}u = \frac{b_1}{r} \int_0^{\infty} R_{\alpha}(s) {\rm d}s
	+ {\Gamma(1-\alpha)} ~ \ \ell(\infty) \ \int_{0}^{\infty} \left(\f{e^{- r (1 - \rho)s} - e^{-\m s} }{\m - r (1 - \rho)}\right)^\alpha {\rm{d}} u ,
\]
and hence
	\[
	\int_{0}^{\infty} R_{\alpha}(u)  {\rm{d}} u =
  %{\Gamma(1-\alpha)} ~\ \ell(\infty) \E B \int_{0}^{\infty} \left(\f{e^{- r (1 - \rho)s} - e^{-\m s} }{\m - r (1 - \rho)}\right)^\alpha {\rm{d}} s  \int^{\infty}_{u = s} e^{-r (1 - \rho)(u - s)}    {\rm{d}} u\nonumber \\
	%&+ {\Gamma(1-\alpha)} ~ \ \ell(\infty) \ \int_{0}^{\infty} \left(\f{e^{- r (1 - \rho)s} - e^{-\m s} }{\m - r (1 - \rho)}\right)^\alpha {\rm{d}} u \nonumber \\
	 \f{\ {\Gamma(1-\alpha)} ~\ \ell(\infty)}{1 - \rho} \int_{0}^{\infty} e^{- \alpha r (1 - \rho)s} \left(\f{1 - e^{r(1 - \rho)s - \m s} }{\m - r (1 - \rho)}\right)^\alpha {\rm{d}} s.
	\]
	Substituting $e^{r(1 - \rho)s - \m s} := v$ in the above equation, yields
	\begin{align*}\label{GT_HtafficH8}
	\int_{0}^{\infty} R_{\alpha}(u)  {\rm{d}} u &=   \f{\ {\Gamma(1-\alpha)} ~\ \ell(\infty)}{(1 - \rho)(\m - r (1 - \rho))^{\alpha + 1}} \int_{0}^{1} v^{\f{\alpha r (1 - \rho)}{\m - r (1 - \rho)} - 1} \left( 1 - v\right)^\alpha {\rm{d}} v \nonumber \\
	&= \f{\ {\Gamma(1-\alpha)} ~\ \ell(\infty)}{(1 - \rho)(\m - r (1 - \rho))^{\alpha + 1}}~ \textbf{B}\left(\f{\alpha r (1 - \rho)}{\m - r (1 - \rho)}, ~ \alpha + 1\right),
	\end{align*}
	where $\textbf{B}(\cdot, \cdot)$ is the Beta function. Using a well-known property of the Beta function, viz., $p\,\textbf{B}(p, q) = (p+q) \textbf{B}(p + 1, q)$, we obtain
	\[
	\int_{0}^{\infty} R_{\alpha}(u)  {\rm{d}} u =  \f{\ {\Gamma(1-\alpha)} ~\ \ell(\infty)}{(1 - \rho)(\m - r (1 - \rho))^{\alpha + 1}}~ \f{\m (\alpha + 1) - r (1 - \rho)}{\alpha r (1 - \rho)}~\textbf{B}\left(\f{\alpha r (1 - \rho)}{\m - r (1 - \rho)} + 1, ~ \alpha + 1\right).
	\]
\end{remark}

\subsection{Heavy-traffic asymptotics}
%%%%%%%%%%%%%%%%%%%%%%%%%%%%% Heavy  traffic    %%%%%%%%%%%%%%%%%%%%%%%%%%%%%%%%%%%%%%
\label{sec5.2}

In this subsection we discuss the heavy-traffic behavior of $\Lambda \equiv \Lambda(\infty)$ and $N \equiv N(\infty)$,
i.e., we consider the stationary Hawkes intensity process and the number of customers, in the regime where we let the load generated by the Hawkes process approach its instability boundary. In other words, we study the system's behavior when  $\rho= {b_1} \int_0^{\infty} h(u) {\rm d}u \uparrow 1$.  Let ${\Gamma}(\a_1,\a_2)$ denotes a Gamma distribution with shape parameter $\a_1$ and rate parameter $\a_2$. We assume in this subsection that $h(u)=e^{-ru}$, so that $\rho = {b_1}/r$.

\begin{theorem}\label{Theo:HT_LT}
	Consider a Hawkes process for which the first two moments of $B$ are finite,
and where $h(u) = {\rm e}^{-ru}$, $u \geqslant 0$.
As ${\rho \uparrow 1}$,
	\begin{equation}
	(1 - \rho) \Lambda \stackrel{d}{\to} {\Gamma}\left(\f{2 r \l_{\infty} }{b_2}, \f{2r}{b_2}\right).
	\end{equation}				
\end{theorem}

\begin{proof}
	Start with the steady-state version of \eqref{Eq: Transiernt1}. Take $z=$1, so that we focus on $\Lambda$, as  $\E e^{-s \Lambda}$  is the LST of $\Lambda$. 	Then one gets the ODE:
	
	$$ \f{{\rm d}}{{\rm d} s} \E e^{-s \Lambda} = - \f{r s \lambda_{\infty}}{rs + \beta(s) -1} \E e^{-s  \Lambda}.$$
	
	Its solution is	
\begin{equation}\label{LST_Lambda-infty}
	\E e^{-s \Lambda} = \exp\left[- \lambda_{\infty} r \int_0^s \f{u}{ru + \beta(u) -1} {\rm d}u \right].
\end{equation}	
	
	Since we assume that the first two moments of $B$ are finite, we can write $\beta(s) = 1 - s \ {b_1}  +  \f{s^2}{2} b_2  + o(s^2),$ as $s  \downarrow 0$. Now consider $\E e^{-s (1 - \rho) \Lambda}$ with $\rho = b_1 /r$. Substituting  $u = v (1-\rho)$ and the above $\beta(s)$ expansion in \eqref{LST_Lambda-infty}, we get for $\rho \uparrow 1$:
	
%	\begin{align} 
%	\E e^{-s (1 - \rho) \Lambda} &= \exp\left[- \lambda_{\infty} \int_0^s \f{1}{1 + vb_2/(2r) - b_3/(6r) v^2 (1-\rho) + O(1 - \rho)} {\rm d} v \right].
%	\end{align}
\begin{align} 
\E e^{-s (1 - \rho) \Lambda} &= \exp\left[- \lambda_{\infty} \int_0^s \f{1}{1 + vb_2/(2r) + o(1 - \rho)} {\rm d} v \right].
\end{align}
	For $\rho \uparrow 1$, one gets from the above equation that, with $\kappa= 2r/b_2$,
	\begin{align} 
	\lim_{\rho \uparrow 1} \E e^{-s (1 - \rho) \Lambda} = \left(\f{\kappa}{\displaystyle \kappa +  s}\right)^{\kappa \lambda_{\infty}}.
	\end{align} 
		By virtue of L\'evy's convergence theorem \cite[Ch. XVIII]{WILL}, the result follows.		
\end{proof}
We now prove that a very similar result holds for the steady-state number of customers $N$ in the $\textrm{Hawkes}/M/\infty$ queue.
We accomplish this by first observing that the first two moments of $(1-\rho) \Lambda$ and $(1-\rho) \mu N$ have the same limit for $\rho \uparrow 1$
(in fact, this holds for the first $g$ moments, if $\E B^g < \infty$, as can easily be verified from the steady-state version of
(\ref{Eq: Transiernt2})).
We subsequently apply the following lemma.
Denote $\sigma^2_{X_n} := \mathbb{V}{\rm ar}\, X_n$, $\sigma^2_{Y_n} := \mathbb{V}{\rm ar}\, Y_n$, and $C_n := \mathbb{C}{\rm ov}(X_n, Y_n)/\sqrt{\mathbb{V}{\rm ar}\, X_n\cdot \mathbb{V}{\rm ar}\, Y_n}$. 
	\begin{lemma}
		Suppose that the following conditions hold:
		(i) $ \E Y_n/\E X_n\to \b$ and $\sigma_{Y_n}/\sigma_{X_n} \to \b$ as $n\to\infty$, (ii)
			$C_n \to 1$ as $n \to \infty$, (iii)
			$X_n \stackrel{d}{\to} X$, and (iv) there is a finite $M$ such that, for all $n$, $0\leqslant {\mathbb E}\,Y_n<M$ and 
			$0\leqslant \sigma^2_{Y_n} < M$. 
				Then $Y_n \stackrel{d}{\to} \b X$.
	\end{lemma}

	\begin{proof}
		Define $\a_n :=  {C_n \ \sigma_{Y_n}}/{\sigma_{X_n}}$. Then by using (i) and (ii),  $\a_n \to \beta$. In addition,{
		\begin{align}\label{Var_alpha_eq}
		\mathbb{V}{\rm ar}(\a_n X_n - Y_n) &= \a_n^2 \sigma^2_{X_n} - 2 \a_n C_n  \sigma_{X_n} \sigma_{Y_n} +\sigma^2_{Y_n} \nonumber \\
		&=C^2_n \sigma^2_{Y_n} - 2C_n^2 \sigma^2_{Y_n} + \sigma^2_{Y_n} 
		 = (1-C_n^2) \sigma^2_{Y_n} \to 0,
		\end{align}}
		using (ii) and (iv). Then note that 
		\[
		\E (\a_n X_n - Y_n)^2 = \mathbb{V}{\rm ar}(\a_n X_n - Y_n) + (\a_n \E X_n - \E Y_n)^2\to 0;
		\]
		the first term goes to 0 due to  \eqref{Var_alpha_eq}, and the second term due to  
		(i), (iv), and 
		$\a_n \to \b$ as $n \to \infty$. Hence  $\a_n X_n - Y_n \to 0$ in ${L^2}$; using Chebyshev's inequality, it immediately follows that $\e_n :=  \a_n X_n - Y_n \stackrel{P}{\to} 0 $. 
		
		Now consider $Y_n = \a_n X_n - \e_n.$  Using (a) 	$\a_n \to \b$ as $n \to \infty$, (b) $X_n \stackrel{d}{\to} X$ because of (iii),
			(c) $\e_n \stackrel{P}\to 0$,
			(d) Slutsky's Lemma \cite[Lemma 2.8]{vdvaart}: $A_n \stackrel{d}{\to} A, B_n \stackrel{P}{\to} 0$ implies that $A_n + B_n \stackrel{d}{\to} A$ (irrespective of $A_n$ and $B_n$ being dependent),
				the claim $Y_n\stackrel{d}{\to} \beta X$  follows.
	\end{proof}
This lemma, in combination with Theorem~\ref{Theo:HT_LT}, now yields the following heavy-traffic result for $N$, the number of customers in the $\textrm{Hawkes}/M/\infty$ queue.
For this purpose, we take $X_n = (1 - \rho_n) \Lambda_n$ and $Y_n = (1 - \rho_n) N_n$, where $\rho_n$, $n=1,2,\dots$ is a sequence of parameters converging to $1$,
and where $\Lambda_n$ and $N_n$ correspond to the quantities $\Lambda$ and $N$ in our $\textrm{Hawkes}/M/\infty$ queue when $\rho=\rho_n$.
\begin{theorem}\label{Theo:Heavy_traffic_lighttailN}
Consider the $\textrm{Hawkes}/M/\infty$ queue, with ${\mathscr J }( u) =  e^{-\m u}$.
Assume that the first two moments of $B$ are finite, and that $h(u) = {\rm e}^{-ru}$, $u \geqslant 0$.
As ${\rho \uparrow 1}$,
	\begin{equation}
	(1 - \rho) N \stackrel{d}{\to} {\Gamma}\left(\f{2 r \l_{\infty} }{b_2}, \f{2r \m}{b_2}\right).
	\end{equation}				
\end{theorem}

\section{Discussion and concluding remarks}
\label{sec:Discussion}
In this paper we have analyzed an infinite-server queue fed by a Hawkes arrival process. Under Markovian assumptions, a fairly explicit analysis is possible, leading to e.g.\ explicit expressions for the (transient and stationary) moments of the number of customers in the system. Lifting the Markovian assumptions, the analysis becomes less explicit: results are derived in terms of  a fixed-point equation describing the $z$-transform of the number of customers. We have used this fixed-point equation to derive asymptotic results. 

Several branches of follow-up research offer {themselves}.
\begin{itemize}
\item[$\circ$]
In the first place one could consider {\it single-server} queues with Hawkes input. In our analysis in the infinite-server setting we repeatedly use that the customers are served independently of each other, a property that we do not have in a single-server context. This may entail that exact analysis is prohibitively difficult, but analysis in a heavy-traffic setting might be possible.
\item[$\circ$]
In \cite{Koops2017} we considered networks of infinite-server queues with shot-noise-driven input. There it turned out that the network setting could be analyzed by essentially the same techniques as the single-queue setting. This raises the question whether the results of the present paper also naturally extend to that of a {\em network} of infinite-server queues with Hawkes input. We would have to appeal to the multivariate counterpart of self-exciting processes, which are called \textit{mutually-exciting} arrival processes \cite{mutuallyexciting}. As the name suggests, in this case arrivals to a particular queue are able to excite arrivals in other queues.
\item[$\circ$] { There is a vast body of literature that deals with statistical inference for Hawkes processes, which implies that Hawkes processes can be fitted to real data. Most extensive practical applications in the literature are related to limit order book data, cf.\ e.g.\ \cite{Bacry, Toke,Blanchet2017}. In case of Markovian Hawkes processes, the likelihood function has an explicit expression \cite{OgataAkaike1982}, and the resulting maximum likelihood estimators are known to be consistent, asymptotically normal and efficient \cite{Ogata1978}. There are also several more recent papers that consider nonparametric estimation of non-Markovian Hawkes processes, e.g.\ \cite{Kirchner}. For a concise survey on statistical literature on Hawkes processes, see \cite[Appendix C]{Bacry}. The paper \cite{Blanchet2017} empirically supports the use of infinite-server queues in limit order book models. An interesting line of further research is to verify if predictions derived from the model in this paper correspond to empirical order book data as well.}
\item[$\circ$]One could pursue improving the bounds of Prop.\ \ref{prop:upperlower}, to make them more useful for larger values of $n$. Also better initial approximations can be determined (see Remark \ref{rem:alg}).
\item[$\circ$]
In the heavy-traffic setting, there are several interesting directions still to be explored. One could try to generalize Theorem~\ref{Theo:Heavy_traffic_lighttailN} to the case of generally distributed $J$;
it would also be interesting to study the heavy-traffic behavior of $\Lambda$ 
and $N$ in the heavy-tailed setting of Section \ref{sec5.1}.
\end{itemize}

\subsubsection*{Acknowledgments}
{\footnotesize After the submission of the ArXiv version of this paper (ref.\ 1707.02196; July 7th, 2017),  J.\ Pender (Cornell University) kindly informed us
that A.\ Daw and he were also completing a paper \cite{Pender} on infinite-server queues with Hawkes arrival processes. We thank P.\ Spreij (University of Amsterdam) for useful discussions.}

{\footnotesize The research for this paper is partly funded by the NWO Gravitation Project NETWORKS, Grant Number 024.002.003 (Boxma, Koops, Mandjes) and an NWO Top Grant, Grant Number  613.001.352 (Boxma, Mandjes, Saxena). The research of O.\ Boxma was also partly funded by the Belgian Government, via the IAP Bestcom Project.}
 
{\footnotesize
\bibliographystyle{acm}
\bibliography{biblio}}
\end{document}